\RequirePackage{fix-cm}
\documentclass[envcountsect]{svjour3}        

\smartqed  
\usepackage{graphicx}
\usepackage{amsmath}
\usepackage{amssymb}
\usepackage{ntheorem}
\usepackage{enumerate}
\usepackage[all,knot,arc,import,poly]{xy}
\usepackage[usenames]{color}
\usepackage{hyperref}
\hypersetup{
    colorlinks=true,
    linkcolor=blue,
    filecolor=magenta,      
    urlcolor=cyan,
    pdftitle={},
    pdfpagemode=FullScreen,
    }

\newenvironment{prooof}{{\bf Proof:}}{\hfill$\square$}

\spnewtheorem{Df}{Definition}[section]{\bfseries}{\upshape}
\spnewtheorem{Teo}[Df]{Theorem}{\bfseries}{\upshape}
\newtheorem*{TeoNoNr}{Theorem}   
\spnewtheorem{Prop}[Df]{Proposition}{\bfseries}{\upshape}
\spnewtheorem{Lem}[Df]{Lemma}{\bfseries}{\upshape}
\spnewtheorem{Obs}[Df]{Remark}{\bfseries}{\upshape}
\spnewtheorem{Fact}[Df]{Fact}{\bfseries}{\upshape}
\spnewtheorem{Fat}[Df]{Claim}{\bfseries}{\upshape}
\spnewtheorem{Que}[Df]{Question}{\bfseries}{\upshape}
\spnewtheorem{Cor}[Df]{Corollary}{\bfseries}{\upshape}
\spnewtheorem{Ex}[Df]{Example}{\bfseries}{\upshape}
\spnewtheorem{Notation}[Df]{Notation}{\bfseries}{\upshape}
\spnewtheorem{construction}[Df]{Construction}{\bfseries}{\upshape}

\newcommand{\N}{\mathbb{N}}

\def\k{\kappa}

\newcommand{\Fm}[2]{{Fm_{#1}(#2)}}   
\newcommand{\Fi}[2]{{Fi_{#1}(#2)}}   

\newcommand{\FilPa}{\mathcal{F}}          

    \newcommand{\FP}[1]{\mathbf{FP}_{#1}}          
\newcommand{\FPm}[1]{\FP{#1}^{\text{mono}}}          
\newcommand{\FPi}[1]{\FP{#1}^{\text{incl}}} 
\newcommand{\freeFP}[1]{\textrm{free-}\mathbf{FP}_{#1}}          
\newcommand{\freeFPm}[1]{\freeFP{#1}^{\text{mono}}}          

\newcommand{\Th}{\textrm{Th}} 
\newcommand{\Log}[1] {\mathcal{L}_{#1}} 

\newcommand{\AL}{\mathbf{AL}}

\newcommand{\CLat}{\mathbf{CLat}}

\newcommand{\Lk}{\mathbf{Lat_\kappa}}

\newcommand{\SigmaStr}{\Sigma\text{-}\mathbf{Str}}

\newcommand{\Tb}{\mathbb{T}}

\newcommand{\Cc}{\mathcal{C}}

\DeclareMathOperator{\card}{card}
\DeclareMathOperator{\id}{id}

\begin{document}

\title{Filter pairs and natural extensions of logics
}


\author{Peter Arndt \and
        Hugo Luiz Mariano \and
        Darllan Conceição Pinto 
}


\institute{P. Arndt \at
              Heinrich-Heine-Universit\"at D\"usseldorf-Germany \\
              \email{peter.arndt@uni-duesseldorf.de}           
           \and
           H. L. Mariano \at
              University of S\~ao Paulo-Brazil\\
              \email{hugomar@ime.usp.br} 
           \and
           D. C. Pinto \at
              Federal University of Bahia-Brazil\\
              \email{darllan@ufba.br}\\
              This author was funded by FAPESB, Grant APP0072/2016.
}

\date{Received: date / Accepted: date}

\maketitle

\begin{abstract}
We adjust the notion of finitary filter pair, which was coined for creating and analyzing finitary logics, in such a way that we can treat logics of cardinality $\kappa$, {where $\k$ is a regular cardinal}. The corresponding new notion is called $\kappa$-filter pair. A filter pair can be seen as a presentation of a logic, and we ask what different $\kappa$-filter pairs give rise to a fixed logic of cardinality $\kappa$. To make the question well-defined we restrict to a subcollection of filter pairs and establish a bijection from that collection to the set of natural extensions of that logic by a set of variables of cardinality $\kappa$.

 Along the way we use $\kappa$-filter pairs to construct natural extensions for a given logic, work out the relationships between this construction and several others proposed in the literature, and show that the collection of natural extensions forms a complete lattice.
 
 In an optional section we introduce and motivate the concept of a general filter pair.




%

\end{abstract}

\section{Introduction}
\label{intro}
In this work we adjust the notion of finitary filter pair from \cite{AMP}, which was coined for creating and analyzing finitary logics, in such a way that we can treat non-finitary logics.

\medskip

{\bf Filter pairs:} 
In \cite{AMP} the notion of finitary filter pair was introduced. The starting point for this definition was the fact that for every finitary logic, with set of formulas $Fm$, the lattice of theories is an algebraic lattice contained in the powerset, $\text{Th} \subseteq \wp(Fm)$, and this lattice completely determines the logic. This lattice is closed under arbitrary intersections and directed unions. The structurality of the logic means that the preimage under a substitution of a theory is a theory again or, equivalently, that the following diagram commutes for every substitution $\sigma$, seen as an endomorphism of the algebra of formulas:

$$\xymatrix{
Fm \ar[d]_\sigma & \text{Th} \ar@{^(->}[r]^-i & \wp(Fm) \\
Fm & \text{Th} \ar[u]^{\sigma^{-1}|_{_{Th}}} \ar@{^(->}[r]^-i & \wp(Fm) \ar[u]_{\sigma^{-1}}
}$$
This says that the inclusion of theories into the power set is a natural transformation, in the sense of category theory. 

Passing from just the formula algebra to arbitrary $\Sigma$-structures (where $\Sigma$ is the signature of the logic), the role of theories can be replaced by the more general notion of filter. The corresponding considerations then apply: preimages of filters under homomorphisms of $\Sigma$-structures are filters again, and this can be rephrased as saying that the inclusions of filters into the full power sets of $\Sigma$-structures form a natural transformation.

Finally, replacing the lattice of filters with a more abstract lattice, we arrive at the notion of finitary filter pair: In \cite{AMP} a \emph{finitary filter pair} over a signature $\Sigma$ was defined to be a pair $(G,i)$, where $G \colon \SigmaStr^{op} \to \AL$ is a functor from $\Sigma$-structures to algebraic lattices and $i$ is a natural transformation from $G$ to the contravariant power set functor $\SigmaStr^{op} \to \AL, \ A \mapsto \wp(A)$. The transformation $i$ is required to preserve, objectwise, arbitrary infima and directed suprema.

The intuition offered in \cite{AMP} about the notion of filter pair was that it is a \emph{presentation} of a logic, different in style from the usual presentations by axioms and rules or by matrices. Instead, it is a direct presentation of the lattice of theories as the image of a map of ordered sets. The required properties ensure that the image really is the theory lattice of a finitary logic. 
A filter pair can provide useful structure for analyzing the associated logic.
In Section \ref{SectionFilterPairs} below we introduce filter pairs more carefully and indicate some of their uses.

\medskip

{\bf ${\kappa}$-filter pairs:}
In the concept of finitary filter pair, the cardinality $\aleph_0$ is hidden in the notions of \emph{directed supremum} and \emph{algebraic lattice}: Recall that a subset $S$ of a poset $P$ is \emph{directed} if any finite set of elements of $S$, i.e. any set of cardinality smaller than $\aleph_0$, has a supremum in $S$. An element is \emph{compact} if, whenever it is smaller than or equal to the supremum of a directed set, it is smaller than or equal to one of the members of the directed set. A complete  lattice is \emph{algebraic} if every element is a supremum of compact elements.
An inspection of the proofs of \cite{AMP} shows that it is the condition that $i$ preserves directed suprema, that implies that the associated logic is finitary.

\emph{Our first aim in this article is to introduce $\kappa$-filter pairs}, a generalization of finitary filter pairs that allows to treat logics of all cardinalities. Here the cardinality of a logic is the smallest infinite cardinal $\kappa$ such that whenever $\Gamma \vdash \varphi$ holds for some formulas, one finds a subset $\Gamma' \subseteq \Gamma$ of cardinality strictly smaller than $\kappa$ such that $\Gamma' \vdash \varphi$. Thus finitary logics are the logics of cardinality $\aleph_0$.

The notion of $\kappa$-filter pair arises by replacing the (implicit) occurrences of the cardinal $\aleph_0$ in the definition of filter pair by a regular cardinal $\kappa$ in an appropriate way, see Definition \ref{DefinitionKappaFilterPair}. Doing this, one can show in a similar way as that of \cite{AMP} that $\kappa$-filter pairs give rise to logics of cardinality $\leq \kappa$ (Prop. \ref{PropLogicsFromFilterPairs}) and that vice versa every logic of cardinality $\kappa$ arises from a $\kappa$-filter pair (Theorem \ref{TheoremEveryLogicComesFromAFilterPair}). Concretely, this is achieved by the so-called \emph{canonical filter pair} of the logic, where for a $\Sigma$-structure $A$ the lattice $G(A)$ is given by the collection of all $l$-filters on $A$, and $i_A$ is the inclusion into the power set.


\medskip

{\bf Filter pairs giving rise to a fixed logic:} 
The intuition of a filter pair as a presentation of a logic raises the question which different filter pairs give rise to the same fixed logic
. \emph{The study of this question is our second aim in this article.}

The collection of all these filter pairs, even up to an appropriate notion of isomorphism, forms a proper class
(see Remark \ref{RemarkProperClassOfEquivalentFilterPairs}), and a classification is not only hopeless, but would also be unilluminating.

To obtain a meaningful parametrization of all filter pairs giving rise to a fixed logic, one has to review, what precisely one wants a filter pair to be a presentation of. From a filter pair  we can not only extract a logic, but also a structure of generalized matrix for every $\Sigma$-structure. Thus a filter pair can be seen as presentation of a coherent family of generalized matrices, tied together by the naturality assumption. 
If we identify two filter pairs that give back precisely the same generalized matrices, then each equivalence class contains a unique filter pair $(G,i)$ for which the maps $i_A$ are injective, see Remark \ref{RemarkProperClassOfEquivalentFilterPairs}. Such a filter pair is called \emph{mono filter pair}.
Thus mono filter pairs correspond to coherent systems of generalized matrices. The classification of such systems is an interesting, but still very difficult question.

If we further identify two filter pairs if they give back the same \emph{logics} -- i.e. if we only look at the lattices associated to absolutely free algebras -- the equivalence classes are in bijection with the so-called \emph{free mono filter pairs}. We are able to classify the free mono filter pairs giving rise to a fixed logic, in terms of the \emph{natural extensions} (see the next paragraph for this notion) of that logic:
\begin{TeoNoNr}\emph{(Theorem \ref{TeoLatticeIsoBetweenFreeMonoFilterPairsAndNatExts})}
Let $l$ be  logic of cardinality $\kappa$. The free mono filter pairs presenting $l$ are in bijection with the natural extensions of $l$ to a set of variables of cardinality $\kappa$.
 \end{TeoNoNr}

Since for finitary logics the lattice of natural extensions is trivial, this is a genuinely new aspect arising for $\kappa$-filter pairs.


\medskip

{\bf Natural extensions:} 
A natural extension of a logic $l=(\Fm{\Sigma}{X}, {\vdash})$ to a set of variables $Y$ is a conservative extension to $\Fm{\Sigma}{Y}$ which has the same cardinality as $l$. {This notion appears in some proofs of transfer theorems in Abstract Algebraic Logic.} Clearing up some misconceptions from the literature, Cintula and Noguera \cite{CintulaNoguera} showed that a certain proposed construction of a natural extension could fail. They gave sufficient conditions for the existence and uniqueness of natural extensions, and asked whether there always exists a unique natural extension to a given set of variables. Shortly after, P{\v r}enosil \cite{Prenosil} gave two constructions of natural extensions \emph{of logics whose cardinality is a regular cardinal $\kappa$}, a maximal and a minimal one. He also gave further results on proposed solutions from the literature, and showed that there can be several different natural extensions of a logic, answering the uniqueness question in the negative. 

\medskip

In Corollary \ref{CorollaryKappaFilterPairPresentingLogicOfCardinalityKappaAlsoYieldsNaturalExtensions} we show that a filter pair provides \emph{natural extensions} of a logic to all sets of variables, thus giving an alternative proof for the existence of natural extensions for logics of regular cardinality. 

In Remark \ref{ObsPorqueTratamosDoCasoRegular} we explain where the regularity assumption, left implicit by P{\v r}enosil, enters. The question of the existence of natural extensions for logics of singular cardinality remains open, but in Corollary \ref{CorExtNatCasoSingular} we give the next best solution, showing that there are conservative extensions whose cardinality is the next regular cardinal.

In the literature one finds tentative constructions of consequence relations ${\vdash}^{\textrm{\tiny {\L}S}}, {\vdash}^{\textrm{\tiny SS}}, {\vdash}^-$ and ${\vdash}^+_\kappa$ of which the first one can fail to be structural, the second one can fail to satisfy idempotence and the last are the two are the minimal, resp. maximal, natural extensions found by P{\v r}enosil in \cite{Prenosil}. We summarize the definitions and the known interrelations between these proposed solutions at the beginning of Section 5.

We identify the natural extension given by the canonical filter pair with \linebreak P{\v r}enosil's minimal one, and complete the picture painted by Cintula, Noguera and P{\v r}enosil in the following result.

\begin{TeoNoNr}\emph{(Theorem \ref{TheoremInclusionsBetweenTheRelations})}
 Given sets $X \subseteq Y$ of variables and a logic \linebreak 
 $l=(\Fm{\Sigma}{X}, {\vdash})$, we have the following inclusions between the associated relations on $\Fm{\Sigma}{Y}$:
 $${\vdash}^{\textrm{\tiny {\L}S}} \ \ \subseteq \ \ {\vdash}^{\textrm{\tiny SS}} \ \ \subseteq \ \ {\vdash}^-  \ \ = \ \ {\vdash}_{\Log{Y}(\FilPa(l))}  \ \ \subseteq \ \ {\vdash}^+_\kappa.$$
The second relation is the structural closure of the first one and the third is the idempotent closure of the second one.
 \end{TeoNoNr}

\medskip 

While the topic of natural extensions has some technical importance, we largely agree with P{\v r}enosil that in considering a particular logic for a concrete purpose one can, in probably all cases, just endow it from the beginning with sufficiently many variables to escape the questions about existence and uniqueness.
Our interest in natural extensions in this article is mainly that they give a solution to the ``reverse engineering'' question of parametrizing all filter pairs which present a given logic.

As a byproduct of the discussion we obtain the following new result, which is of independent interest:

\begin{TeoNoNr}\emph{(Corollary \ref{natextlattice})}
Let $l$ be a logic of regular cardinality $\kappa$. The set of natural extensions of $l$ with respect to a fixed set of variables of cardinality $\kappa$, ordered by deductive strength, is a complete lattice.
\end{TeoNoNr}

\medskip

{\bf Overview of the article:} In section \ref{SectionPreliminaries} we collect the necessary notions and standard results. In Section \ref{SectionFilterPairs} we introduce general filter pairs and motivate their study by some examples of application. In Section \ref{SectionFilterFunctors} we introduce $\kappa$-filter pairs. In Section \ref{SectionNaturalExtensions} we discuss natural extensions and how filter pairs give rise to them. 
In Section \ref{SectionFilterPairsYieldingAFixedLogic} we investigate the collection of filter pairs that yield a fixed logic. We finish with some final remarks in Section \ref{SectionFinalRemarks}.

Apart from Section \ref{SectionFilterPairsYieldingAFixedLogic}, the sections can be read independently to some extent, possibly referring back to Section \ref{SectionPreliminaries} for definitions:
The reader who wishes to just get an impression of what filter pairs are about, can simply read Section \ref{SectionFilterPairs}. A reader who is familiar with finitary filter pairs and wants to learn about $\kappa$-filter pairs, can directly jump to \ref{SectionFilterFunctors}. A reader who only wants to learn about the relations between the classical tentative constructions of natural extensions can directly jump to Section \ref{SectionNaturalExtensions} and ignore the mentions of filter pairs there.

\medskip

{\bf Acknowledgments:} We thank the referee for the detailed remarks, which greatly improved the readability of the article.

%
%

\section{Preliminaries}\label{SectionPreliminaries}

{In this section, we recall the basic definitions and results on logic, closure operators and complete lattices and their relative versions associated to an infinite cardinal $\kappa$.}

\begin{Df}\label{DefConsequenceRelation}
A signature is a sequence of pairwise disjoint sets $\Sigma=(\Sigma_{n})_{n\in \N}$. The set $\Sigma_{n}$ is called the set of $n$-ary connectives. For a set $X$ we denote by $\Fm{\Sigma}{X}$ the absolutely free algebra over $\Sigma$ generated by $X$, also called the set of formulas with variables in $X$.

A consequence relation is a relation ${\vdash}\subseteq\wp(\Fm{\Sigma}{X})\times \Fm{\Sigma}{X}$, on a signature $\Sigma=(\Sigma_{n})_{n\in \N}$, such that, for every set of formulas $\Gamma,\Delta$ and every formula $\varphi,\psi$ of $\Fm{\Sigma}{X}$, it satisfies the following conditions:
\begin{itemize}
\item[$\circ$]{\bf Reflexivity:} If $\varphi\in\Gamma,\ \Gamma\vdash\varphi$
\item[$\circ$]{\bf Cut:} If $\Gamma\vdash\varphi$ and for every $\psi\in\Gamma,\ \Delta\vdash\psi$, then $\Delta\vdash\varphi$
\item[$\circ$]{\bf Monotonicity:} If $\Gamma\subseteq\Delta$ and $\Gamma\vdash\varphi$, then $\Delta\vdash\varphi$
\item[$\circ$] {\bf Structurality:} If $\Gamma\vdash\varphi$ and $\sigma$ is a substitution\footnote{I.e.  $\sigma \in hom_{\Sigma}(\Fm{\Sigma}{X},\Fm{\Sigma}{X})$.}, then $\sigma(\Gamma)\vdash\sigma(\varphi)$
\end{itemize}
\end{Df}

The notion of logic that we consider is the following:

\begin{Df}
A logic is a triple $(\Sigma, X, {\vdash})$ where $\Sigma$ is a signature, $X$ is an {{{\bf infinite}}} set of variables and ${\vdash}$ is a consequence relation on $\Fm{\Sigma}{X}$. We often write a logic as a pair $(\Fm{\Sigma}{X},{\vdash})$, with the datum of the signature and the set of variables combined into that of the formula algebra.
\end{Df}

Note that for the considerations in this article the set of variables needs to be part of the definition of logic.

\begin{Df}
A {\bf closure operator} in a set $A$ is a function $c : P(A) \to P(A)$ that is inflationary, increasing and idempotent. We denote by $({\cal C}(A), \leq)$ the poset of closure operators in $A$ ordered setwise by inclusion.
\end{Df}

We will freely switch between the two formulations of a logic as a consequence relation on the set of formulas and as a certain closure operator on that set. The properties of Def. \ref{DefConsequenceRelation} translate to the operator $C_{{\vdash}}\colon \Gamma \mapsto \{\varphi \mid \Gamma \vdash \varphi\}$ being increasing, idempotent, order preserving and structural, respectively.


\begin{Df}\label{DefTheoryFilterLatticeOfFilters}
Let $\Sigma$ be a signature,  $l=(\Sigma, X, {\vdash})$ be a logic and $M$ a $\Sigma$-algebra. 

\begin{itemize} 

\item A subset $T$ of $Fm_\Sigma(X)$ is an {\bf $l$-theory} if for every $\Gamma\cup\{\varphi\}\subseteq Fm_\Sigma(X)$ such that $\Gamma\vdash \varphi$, if $\Gamma\subseteq T$ then $\varphi\in T$. 
Equivalently, an $l$-theory is a ${\vdash}$-closed subset of $Fm_\Sigma(X)$.

\item A subset $F$ of $M$ is an {\bf $l$-filter on $M$} if for every $\Gamma\cup\{\varphi\}\subseteq Fm_\Sigma(X)$ such that $\Gamma\vdash \varphi$ and every valuation (i.e. $\Sigma$-homomorphism) $v: Fm_\Sigma(X) \to M$, if $v(\Gamma)\subseteq F$ then $v(\varphi)\in F$.

\item A pair $(M,F)$, where $F$ is an $l$-filter on $M$, is called an {\bf $l$-matrix}.


\item We denote the collection of all $\,l$-filters on $M$ by $\Fi{l}{M}$ and  $\iota_l(M) \colon$  $\Fi{l}{M}$  $\hookrightarrow \wp(M)$ denotes the inclusion.

\end{itemize}

\end{Df}

Note that, by structurality, a subset $T \subseteq Fm_\Sigma(X)$ is an $l$-theory iff it is an $l$-filter.



\begin{Df} \label{translation-df}
Let $l = (\Sigma, X, {\vdash}), l' =(\Sigma, X', {\vdash}')$ be  logics over a signature $\Sigma$ and let  $t : Fm_\Sigma(X) \to Fm_\Sigma(X')$ be a  $\Sigma$-homomorphism. 

  $t : l \to l'$ is  a {\bf translation} (respectively, a {\bf conservative translation}) whenever for each $\Gamma \cup \{\varphi\} \subseteq Fm_\Sigma(X)$ we have 

$\Gamma \vdash \varphi \Rightarrow t(\Gamma) \vdash t(\varphi)$ (respectively,
$\Gamma \vdash \varphi \Leftrightarrow t(\Gamma) \vdash t(\varphi)$).

\end{Df}

\begin{Notation}
For a cardinal $\kappa$, we write  $P_{< \kappa}(\Gamma):=\{ \Gamma' \subseteq \Gamma \mid |\Gamma'|<\kappa\}$ for the set of subsets of cardinality smaller than $\kappa$. 
\end{Notation}


\begin{Df}\label{DefCardinalityOfALogic}
Let ${\vdash} \ \subseteq P(A) \times  A$ be a relation between subsets and elements of a given set $A$.

\begin{itemize}
    \item  Let $\kappa$ be an infinite cardinal. The relation ${\vdash}$  is  {\bf $\kappa$-ary}  if for every subset $\Gamma \cup \{\varphi\} \subseteq A$ if $\Gamma \vdash \varphi$ then there exists $\Gamma' \in P_{< \kappa}(\Gamma)$ and $\Gamma' \vdash \varphi$. 
    
    \item  The {\bf cardinality} of a relation ${\vdash}$
is the smallest infinite cardinal $\kappa$ such that ${\vdash}$ is a $\kappa$-ary relation.
\end{itemize}

\end{Df}

\begin{Df}\label{DefKappaAryPartOfALogic}
 Given a relation ${\vdash}$ between subsets and elements of a set, its {\bf $\kappa$-ary part}, ${\vdash}_\kappa$, is defined by $\Gamma \vdash_\kappa \varphi :\Leftrightarrow \ \ \exists \Gamma' \in P_{< \kappa}(\Gamma) \textrm{ with } \Gamma' \vdash \varphi$.
\end{Df}

\begin{Df}\label{DefKappaLogic}
 
 If $l= (\Fm{\Sigma}{X}, {\vdash})$ is a logic such that ${\vdash} \ =  \ {\vdash}_\k$, then $l$ will be called a {\bf $\k$-logic}. This means that $l$ is a logic of cardinality $\leq \k$.
\end{Df}

\begin{Obs}
A logic $(\Fm{\Sigma}{X}, {\vdash})$ is a $\k$-logic if and only if its associated closure operator satisfies $C_{\vdash}(\Gamma) = \bigcup_{\Gamma' \in P_{<\kappa}(\Gamma)} C_{\vdash}(\Gamma')$. In general, such a closure operator is called a $\kappa$-ary closure operator.
\end{Obs}

\begin{Obs} Let $l =(\Sigma, X, {\vdash})$ be a 
logic. Then the following are equivalent:
\begin{itemize}
    \item The $\k$-ary part of $l$ is a logic and  $id : (\Sigma, X,{\vdash}_\k) \to (\Sigma, X, {\vdash})$ is a conservative translation.

\item The logic $l$ is $\k$-ary.

\end{itemize}
\end{Obs}



\begin{Ex} \label{k-logic}
For every infinite cardinal $\kappa$ and any set of variables $X$, there is a logic of cardinality $\kappa$ over $X$.

There is nothing new to add in the case  $\kappa = \aleph_0$. If $\k > \aleph_0$, consider a signature $\Sigma$ by setting $\Sigma_0:=\{ c_\alpha \mid \alpha < \kappa\}$ and $\Sigma_n:=\emptyset$ for $n \geq 1$ -- we just have constant symbols and variables --
thus $\Fm{\Sigma}{X} = \Sigma_0 \cup X$.

Define a logic over $\Fm{\Sigma}{X}$ by taking the closure operator on $\Fm{\Sigma}{X}$  generated by the rules:

Let $\gamma < \k$ be a limit ordinal, $\gamma > 0$. Then $\{c_{\alpha+1} \mid \alpha < \gamma\} \vdash c_\gamma$. Thus, for each $\Gamma \subseteq \Fm{\Sigma}{X}$:

\begin{center}
$C_{\vdash}({\Gamma}) = \Gamma\cup \{c_\gamma \mid \gamma$ is a limit ordinal, $ 0 < \gamma <\k$ and $\forall \alpha < \gamma$, $c_{\alpha+1} \in \Gamma\}$.
\end{center}

This determines a logic with cardinality exactly $\kappa$. Since $c_\gamma$, for any limit ordinal $\gamma< \k$ that is not a cardinal can be derived by a minimal (non-trivial) set of hypotheses with cardinality equal to $\card(\gamma) < \k$, the cardinality of the logic is $\geq \k$. It is also clear that cardinality of the logic is $\leq \k$, since that is the cardinality of the language.
\end{Ex}

 Recall that an infinite cardinal $\kappa$ is called {\bf regular} if the union of fewer than $\kappa$ sets of cardinality less than $\kappa$ has cardinality less than $\kappa$ again. An infinite cardinal that is not regular is called {\bf singular}.

\begin{Obs} \label{regular-obs}

Let $\kappa$ be a regular cardinal.  Recall the following notions from lattice theory:

$\bullet$ \ 
A subset $S$ of a partially ordered set $P$ is called $\kappa$-directed if every subset of $S$ of cardinality strictly smaller than $\kappa$ has an upper bound in $S$. An example is given by the collection of all subsets of a set which have cardinality smaller than $\kappa$.

$\bullet$ \ In a partially ordered set $P$, an element $x \in P$ is called $\kappa$-\emph{small} if for every $\kappa$-directed subset $D \subseteq P$ one has $x \leq \sup D$ \ iff \ $\exists d \in D \colon x \leq d$  (e.g. the finite sets in a power set are $\aleph_0$-small, a.k.a. compact).

$\bullet$ \ A \emph{$\kappa$-presentable lattice}, or \emph{$\kappa$-algebraic lattice}, is a complete lattice such that every element is the supremum of the $\kappa$-small elements below it (e.g. power sets are $\kappa$-presentable for every $\kappa$). We will denote the category of $\kappa$-presentable lattices and all order preserving functions by $\Lk$.
\end{Obs}

\begin{Df}\label{DefRegularizationOfCardinal}
For each infinite cardinal $\k$, we denote  by $reg(\kappa)$ the least regular cardinal $\geq \kappa$. Note that if $\k$ is a singular cardinal, then $reg(\kappa) = \k^+$, thus, in general,  $reg(\kappa) \in \{\kappa, \kappa^+\}$.
\end{Df}

\begin{Fact}\label{FactForRegularCardinalKappaAryPartIsLogic}

\begin{enumerate}

\item The $\kappa$-ary part of a relation that is reflexive (resp. monotonous, resp. structural concerning some set of endofunctions) has the same property and always is a $\kappa$-ary relation.

    \item If $\kappa$ is a regular cardinal, then  
the $\kappa$-ary part of a relation that determines a closure operator still determines a closure operator. In particular the $\k$-ary part of a logic is a $\k$-logic. 
\end{enumerate}
\end{Fact}

\begin{prooof} (1) The inflationary, increasing and structural properties are easy to see. The $\k$-ary property: $\Gamma \vdash_\k \varphi$, then there is $\Gamma' \in P_{<\k}(\Gamma)$ such that $\Gamma' \vdash \varphi$. Then $\Gamma' \vdash_\k \varphi$.

\vspace{0.2cm}
(2) idempotency or cut:\\ Suppose that $\Delta \vdash_\k \phi$ and $\Gamma \vdash_\k \Delta$.
Then let $\Delta' \in P_{<\k}(\Delta)$ such that $\Delta' \vdash \varphi$ and for each $\delta \in \Delta'$ let $\Gamma_{\delta} \in P_{<\k}(\Gamma)$ such that $\Gamma_\delta \vdash \delta$. Since $\k$ is {\bf regular}, take $\Gamma' := \bigcup\{ \Gamma_{\delta} : \delta \in \Delta'\}$, then $\Gamma' \in P_{<\kappa}(\Gamma)$ and $\Gamma' \vdash \Delta'$. Thus $\Gamma' \vdash \varphi$ and $\Gamma \vdash_\k \varphi$.

$\ $
\end{prooof}

By Fact \ref{FactForRegularCardinalKappaAryPartIsLogic}, if $\kappa$ is a regular cardinal, then the $\kappa$-ary part of the consequence relation of a logic is a logic again. For a singular cardinal $\kappa$ this can fail:

\begin{Ex}\label{ExampleKappaAryPartCanFailToBeLogic}
Let $\kappa$ be a singular cardinal, and let $M_i, \ i\in I$ be a family of pairwise disjoint sets with $|I|< \kappa$, $|M_i|<\kappa$ for all $i\in I$ and $|\bigcup_{i \in I}M_i|=\kappa$.

Consider the signature with $\Sigma_0:=\bigcup_{i \in I}M_i \coprod I \coprod \{\ast\}$ and $\Sigma_n=\emptyset$ for $n\neq 0$. For an enumerable set $X$ of variables, consider the consequence relation ${\vdash}$ on $\Fm{\Sigma}{X}$ generated by the rules 
$$M_i \vdash i \ (i \in I) \ \text{ and }\ I\vdash\ast.$$
By idempotence, this consequence relation satisfies $\bigcup_{i \in I}M_i \vdash \ast$, but no proper subset allows this conclusion. Therefore it has cardinality $>\kappa$. In fact it has cardinality $=\kappa^+$, the successor of $\kappa$, because that is the cardinality of the language.

The $\kappa$-ary part ${\vdash}_\kappa$ of ${\vdash}$ contains all of our generating rules, but not the rule $\bigcup_{i \in I}M_i \vdash \ast$, so it fails to satisfy idempotence and is not a logic.
\end{Ex}

\begin{Obs}
The different behaviour of regular and singular cardinals with respect to $\kappa$-ary parts, and also when taking closures, leads to various regularity assumptions in our results, but also for example in the construction of natural extensions of logics -- see Remark \ref{ObsPorqueTratamosDoCasoRegular} for the latter point.
\end{Obs}

Last, we will recall some notions and results on general closure operators.

\begin{Obs}\label{closure-intersection} {\bf On  general closure operators and complete lattices:} Recall that, for every set $X$:
\begin{itemize}
    
\item A subset $I \subseteq P(X)$ is an {\bf intersection family} iff it is closed under arbitrary intersections (with the convention that empty intersection = $X$). This is the same as the complete lattices $(I, \leq)$ such that the inclusion $\iota : I \hookrightarrow P(X)$ preserves arbitrary infima. We denote by $({\cal I}(X), \subseteq)$ the poset of all intersection families in $X$, ordered by inclusion.

\item It is a well-known result that the mappings below are well defined and  provide a natural anti-isomorphism between the posets $({\cal I}(X), \subseteq)$ and  $({\cal C}(X), \leq)$:

$$I \in {\cal I}(X) \mapsto c_I : P(X) \to P(X), c_I(A) = \bigcap\{C \in I: A \subseteq C\}$$

$$c \in {\cal C}(X) \mapsto I_c = \{ A \in P(X) :  c(A) = A\}.$$

The key points to establish these are:  $c(A)$ is the least $I_c$-closed above $A$  and $$c(\bigcap_{i \in I} A_i) \subseteq \bigcap_{i \in I} c(A_i).$$

\end{itemize}

\end{Obs}

\begin{Obs} Given a regular cardinal $\kappa$, the above correspondence restricts to $\kappa$-ary closure operators (Notation: ${\cal C}_{\kappa}(X)$) and  the $\kappa$-presentable lattices $(I, \leq)$ such that the inclusion $\iota : I \hookrightarrow P(X)$ preserves arbitrary infima and  $\kappa$-directed unions (Notation: ${\cal I}_\kappa(X)$):
The key point to show this is that  $c(\bigcup_{i \in I} A_i) = \bigcup_{i \in I} c(A_i)$ for every $\kappa$-directed union (not only  $c(A) = \bigcup_{A' \in P_{<\kappa}(A)} c(A')$, as in definition). The $\kappa$-compact elements of $I_c$ are exactly the $c(A)$, for each $A \in P_{<\kappa}(X)$. Note that $\{ c(A') : A' \in P_{<\kappa}(A)\}$ is a $\kappa$-directed family of closed subsets, thus $x \notin \bigcup_{A' \in P_{<\kappa}(A)} c_I(A')$ entails $x \notin  c_I(A)$.

\end{Obs}

We present below the explicit calculation of the infima and the relevant suprema  in the posets $({\mathcal C}_\kappa(X), \subseteq)$ that will be useful in Section 4.

\begin{Fact} \label{supinfclosure-fact}

{\underline{Calculation of non-empty infimum}} of  a non-empty family in  $\{c_t : t \in T\} \subseteq {\cal C}_{\kappa}(X)$:

- if $A \in P_{<  \kappa}(X)$, $c(A) := \bigcap_{t \in T} c_t(A)$;

- if $B \in P(X)$, $c(B) := \bigcup_{B' \in P_{<\kappa}(B)} c(B')$.

The key point here (to show idempotence) is realize that if $A \in P_{<  \kappa}(X)$ and $D \in P_{<\kappa}(c(A))$ then, for all $t \in T$, $c_t(D) \subseteq c_t(A)$. 

{\underline{top = inf of the empty family}} in ${\cal C}_{\kappa}(X)$:

$c_\top(A) = X, \forall A \in P(X)$, $I_\top = \{X\}$

{\underline{bottom = sup of the empty family}} in ${\cal C}_{\kappa}(X)$:

$c_\bot(A) = A, \forall A \in P(X)$, $I_\bot = P(X)$

{\underline{Calculation of a $\kappa$-directed sup}} of a upward $\kappa$-directed family in  $\{c_i : i \in (I, \leq)\} \subseteq {\cal C}_{\kappa}(X)$:

- if $A \in P_{<  \kappa}(X)$, $c(A) := \bigcup_{i \in I} c_i(A)$;

- if $B \in P(X)$, $c(B) := \bigcup_{B' \in P_{<\kappa}(B)} c(B')$.







  \end{Fact}

\begin{Fact} \label{k-closure-fa}
Under the notation and hypothesis above, the poset inclusion \linebreak $({\cal C}_\kappa(X), \subseteq)$ $\hookrightarrow ({\cal C}(X), \subseteq)$ has a right adjoint. I.e.: let $c$ be a closure operator. Define $c^{(\kappa)}(A) = \bigcup_{A' \in P_{<\kappa}(A)} c(A')$, $A \in P(X)$. Then $c^{(k)} \in  {\cal C}_\kappa(X)$, $c^{(\kappa)} \leq c$ and, for each $c'  \in  {\cal C}_\kappa(X)$ such that $c' \leq c$, we have $c' \leq c^{(\kappa)}$.

 The only  non-trivial part of the verification is to show that $c^{(\kappa)}$ is an idempotent operator: this follows in the same vein of the construction of $\kappa$-directed suprema in the poset $({\mathcal C}_\kappa(X), \subseteq)$ that we  have described above.

\medskip

\end{Fact}


\section{Filter pairs}\label{SectionFilterPairs}

A (not necessarily structural) logic on $\Fm{\Sigma}{X}$ is the same thing as a closure operator $C$ on a formula algebra $\Fm{\Sigma}{X}$. 
On the other hand, closure operators correspond to intersection families (Remark \ref{closure-intersection}). For instance, a logic $l$ is determined by its  (complete lattice of)  theories, $Th(l)$, and the inclusion of this intersection family $i : Th(l) \hookrightarrow \wp(\Fm{\Sigma}{X})$ gives a "canonical presentation" of the  logic $l = (\Fm{\Sigma}{X}, {\vdash})$.

 More generally,  logics on $\Fm{\Sigma}{X}$ can be obtained from  certain morphisms $i \colon L \to \wp(\Fm{\Sigma}{X})$ from a complete lattice $L$ which preserves arbitrary infima.

The basic idea of filter pairs is to study logics on $\wp(\Fm{\Sigma}{X})$ by presenting their closure operators through convenient morphisms of complete lattices $i \colon L \to \wp(\Fm{\Sigma}{X})$, as above.

Of course, a presentation of this kind is more useful if the lattice $L$ in question is not already given as a lattice of sets of formulas, but for example is a lattice of congruences -- this is the case for the so-called ``congruence filter pairs", defined below.

The next natural question is, what structure or properties to impose on the function $i \colon L \to \wp(\Fm{\Sigma}{X})$ to ensure a given property of the associated logic.

It is not difficult to establish the following characterizations (see, for instance, \cite{AMP}):

\begin{Fact} 

Let $l= (\Fm{\Sigma}{X}, {\vdash})$ be a logic.

\begin{itemize}
    
    \item $l$ is {\bf structural} iff the diagram
$$\xymatrix{
\Fm{\Sigma}{X} \ar[d]_\sigma & Th(l) \ar@{^(->}[r]^-i & \wp(\Fm{\Sigma}{X}) \\
\Fm{\Sigma}{X} & Th(l) \ar[u]^{\sigma^{-1}|_{_{Th}}} \ar@{^(->}[r]^-i & \wp(\Fm{\Sigma}{X}) \ar[u]_{\sigma^{-1}}
}$$
commutes for every $\Sigma$-endomorphism $\sigma : \Fm{\Sigma}{X} \to \Fm{\Sigma}{X}$.
    
    \item $l$ is {\bf finitary} iff  $i: Th(l) \hookrightarrow \wp(\Fm{\Sigma}{X})$ preserves directed suprema.
    
    
\end{itemize}

\end{Fact}

\medskip

Thus structurality of the logic means that the preimage under a substitution of a theory is a theory again. 

Consider the category $\Cc$ consisting of the single  $\Sigma$-algebra $\Fm{\Sigma}{X}$ and all its $\Sigma$-endomorphisms. We have two functors $[\Th], [\wp]\colon \Cc \to \CLat$, into the category of complete lattices and order preserving maps, which assign to the single object $\Fm{\Sigma}{X}$ the lattices $\Th$ and $\wp(\Fm{\Sigma}{X})$, respectively, and to an endomorphism the preimage map $\sigma^{-1}$. Structurality can then be rephrased as saying that the map $i$ is a \emph{natural transformation} from $[\Th]$ to $[\wp]$.

\medskip

These considerations motivate the following definition.

\begin{Df}\label{DefCFilterPair} (adapted from \cite{AMP})
 Let $\Sigma$ be a signature and $\Cc \subseteq \SigmaStr$ a subcategory of the category of $\Sigma$-structures. A $\Cc$-filter pair is a pair $(G,i)$, where $G \colon \Cc^{op} \to \CLat$ is a functor from $\Cc$ to the category of complete lattices and order preserving maps and $i=(i_A \colon G(A) \to \wp(A))_{A \in \Cc}$ is a collection of maps with the following properties:
\begin{enumerate}
 \item For every $A \in \Cc$ the map $i_A$ preserves arbitrary infima.
 
 \item The collection $i$ is a natural transformation from $G$ to the functor $\wp$, sending an algebra to its power set and a homomorphism $\sigma$ to the preimage function $\sigma^{-1}$, i.e. for every $\sigma \colon A \to B \in \Cc$ the following diagram commutes:
  \[
\xymatrix{
A\ar[d]_{\sigma}&G(A)\ar[r]^-{i_{A}}&\wp(A)\\
B&G(B)\ar[u]^{G(\sigma)}\ar[r]_-{i_{B}}&\wp(B)\ar[u]_{\sigma^{-1}}
}
\]
\end{enumerate}
In the case $\Cc = \SigmaStr$ a $\Cc$-filter pair is simply called \emph{filter pair}.
\end{Df}

We observe that the requirement that $i_A \colon G(A) \to \wp(A)$ preserves arbitrary infima implies that it preserves order (since $x \leq y \Rightarrow x = x \wedge y \Rightarrow i_A(x) = i_A(x) \wedge i_A(y) \Rightarrow i_A(x) \leq i_A(y)$) and that it has a left adjoint $\Xi_A \colon \wp(A) \to G(A)$ ($S \in \wp(A) \mapsto \bigwedge\{x \in  G(A) : S \subseteq i_A(x)\}$). It follows that $i_A \circ \Xi_A$ is a closure operator on $A$.

\begin{Df}\label{DefAssociatedLogic}
 Let ${\mathcal F} = (G,i)$ be a $\Cc$-filter pair and let $A$ be an object of $\Cc$. The abstract logic over $A$ associated to ${\mathcal F}$ is the pair $\Log{A}({\mathcal F}) := (A, {\vdash}^{\mathcal F}_A)$, where ${\vdash}^{\mathcal F}_A$ is the consequence relation associated to the closure operator $i_A \circ \Xi_A$.
 
\end{Df}

Moreover, each morphism $\sigma \colon A \to B$ in $\Cc$ induces a translation of abstract logics $\sigma \colon (A, {\vdash}^{\mathcal F}_A) \to (B, {\vdash}^{\mathcal F}_B)$.
We will provide more details on this subject in the next section.

\medskip

The passage from our motivating example of the theory lattice of a logic to Definition \ref{DefCFilterPair} may leave the reader with the following two questions:

\begin{enumerate}
    \item Why consider some subcategory $\Cc$ of all $\Sigma$-structures, and not just the formula algebra ${\Fm{\Sigma}{X}}$?

    \item Why not demand $i$ to be a collection of inclusions? After all the logic \linebreak $\Log{\Fm{\Sigma}{X}}({\mathcal F})$ of Definition \ref{DefinitionLogX} only depends on the image of $i_{\Fm{\Sigma}{X}}$, the lattice of theories.  Indeed, Theorem \ref{TheoremEveryLogicComesFromAFilterPair} below shows that, for presenting a given logic, a filter pair for which $i$ is a collection of inclusions can always be arranged.
\end{enumerate}

    

To illustrate the role of the subcategory $\Cc$, we list a few examples. To separate the two above issues, in all but the last of these examples we suppose the collection $i$ to consist of injective maps -- in this case we speak of a \emph{mono $\Cc$-filter pair}.

\begin{Ex}\label{ExamplesOfGeneralCFilterPairs}
\begin{enumerate}
 \item In case $\Cc$ is the full subcategory of $\SigmaStr$ consisting just of the single object $\Fm{\Sigma}{X}$ and all its endomorphisms, we are back in the situation of the beginning of this section. In this case the datum of a mono $\Cc$-filter pair is equivalent to the datum of a structural logic.
 
 \item In case $\Cc$ is a, possibly non-full, subcategory of $\SigmaStr$ consisting of the object $\Fm{\Sigma}{X}$ and a sub-monoid of the monoid of all its endomorphisms, the datum of a mono $\Cc$-filter pair is equivalent to the datum of a logic that is structural with respect to just the morphisms in $\Cc$.
 
 \item As a special case of the previous example, take $\Cc$ to be the category consisting of the object $\Fm{\Sigma}{X}$ and all endomorphisms which are induced by maps $X \to X$. Then the datum of a mono $\Cc$-filter pair can be understood as the datum of a logic, which is not necessarily fully structural, but for which all variables behave interchangeably.
 
 \item In case $\Cc$ is the category of all free $\Sigma$-structures over infinitely many generators and all homomorphisms between them, we obtain the notion of \emph{free mono filter pair}. By Theorem  \ref{TheoremFilterPairsGiveConservativeExtensions} below, the datum of a free mono filter pair is equivalent to the datum of a structural logic $(\Fm{\Sigma}{X}, {\vdash})$ together with a collection of conservative extensions to the algebras $\Fm{\Sigma}{Y}$ with $Y$ of bigger cardinality than $X$.
  
 \item Continuing the previous example, the datum of a free mono filter pair $(G,i)$ that has the extra property of being a $\kappa$-filter pair - a notion that will  be introduced in the next section - {\em and} that the associated logic $l:=\Log{\Fm{\Sigma}{X}}(G,i)$ (for $X$ an infinite set of variables) has cardinality $\kappa$, is equivalent to the datum of a natural extension of $l$ to a free algebra $\Fm{\Sigma}{Z}$, with a set $Z$ of variables of cardinality $\kappa$ -- this follows from Theorem \ref{TeoLatticeIsoBetweenFreeMonoFilterPairsAndNatExts} below.
 
 \item From a $\Cc$-filter pair $(G,i)$ in the case $\Cc = \SigmaStr$, one obtains an abstract logic over every algebra $A \in \SigmaStr$. As could already be seen in the previous two examples, these abstract logics are not independent from each other, but are tied together by the naturality requirement. By the proof of \cite[{Prop. 2.9}]{AMP} the closed sets (or theories) for these abstract logics are filters for the logic $L:=\Log{\Fm{\Sigma}{X}}(G,i)$ ($X$ any set). Thus the datum of a $\SigmaStr$-filter pair provides a logic, together with a collection of conservative extensions and a structure of generalized matrix on every algebra $A$.
 
 \item A concrete example of a $\Cc$-filter pair with $\Cc = \SigmaStr$ is the so-called canonical filter pair of a logic $l:=(\Fm{\Sigma}{X}, {\vdash})$: This is the filter pair $(Fi_l, i)$ in which $Fi_l(A)$ is the set of $l$-filters on $A$, made into a functor by inverse image, and $i$ is the inclusion in to the power set. By Theorem \ref{TheoremEveryLogicComesFromAFilterPair} below, one has $\Log{\Fm{\Sigma}{X}}(\mathcal{F}(l)) = l$, so that every logic admits a presentation by such a filter pair.
 \end{enumerate}
\end{Ex}




\medskip

We now justify that we don't demand the maps constituting $i$ to be injective, although the point of view of a presentation of a logic just refers to the image of these maps.
More fundamentally, the basic discussions of this section do not {\em yet} provide a justification for the introduction of filter pairs. The answer to both concerns is that the utility of filter pairs ultimately lies in semantic considerations. To illustrate this, we first sketch how some parts of abstract algebraic logic can be rephrased in terms of filter pairs.


\begin{Ex}\label{ExampleCongruenceFilterPairs}
Let $\mathbf{K}$ be a quasivariety of $\Sigma$-structures.  There is a contravariant functor $Co_{\mathbf{K}}$, associating to a $\Sigma$-structure $A$ the lattice of congruences whose quotients lie in $\mathbf{K}$. A filter pair of the form $(Co_{\mathbf{K}}, i)$ is called a congruence filter pair. A canonical recipe to obtain the natural maps $i_A$ is to take a set of equations in one variable $x$: $\tau(x) = \{\langle \delta_i(x), \epsilon_i(x) \rangle: i \in I\}$, and associate to a congruence the elements which solve that equation in the corresponding quotient. More precisely: 

$$\begin{array}{rrcl}
i^\tau_A\colon & Co_{\mathbf{K}}(A) & \to & \wp(A) \\
&                   \theta & \mapsto &  \{a\in A\, \mid\, \forall i \in I \ \langle \delta_i(a), \epsilon_i(a) \rangle \in \theta\}. 
\end{array}
$$

It can be shown that \emph{every} congruence filter pair is of this form, up to isomorphism. In \cite{AMP2} we show:


\begin{center}
\begin{enumerate}
    \item Giving an algebraic semantics of a logic in a quasivariety $\mathbf{K}$ is equivalent to giving a presentation of that logic by a congruence filter pair $(Co_{\mathbf{K}}, i)$.
    \item Giving an \emph{equivalent} algebraic semantics of a logic in a quasivariety $\mathbf{K}$ is equivalent to giving a presentation of that logic by a congruence filter pair $(Co_{\mathbf{K}}, i)$ \emph{where the maps $i_A$ are \emph{injective}}.
\end{enumerate}
\end{center}

As an example, for classical and intuitionistic logics, the single equation \linebreak $\langle \delta_i(x), \epsilon_i(x) \rangle = \langle x, \top \rangle$ generates a filter pair that describes the equivalent algebraic semantics of these Blok-Pigozzi algebraizable logics.


Further properties of filter pairs correspond to other well-known properties of algebraic semantics like regular algebraizability or truth-equationality.

\medskip

Not only can one express standard notions of abstract algebraic logic in the language of filter pairs, one can also prove theorems within this formalism. For example it is a well-known result that for algebraizable logics, the amalgamation property of the quasivariety implies the Craig interpolation property of the logic. In \cite{AMP2} we use the language of congruence filter pairs to prove that the amalgamation property implies Craig interpolation for a class of algebraic semantics encompassing both equivalent algebraic semantics and regular semantics in regular varieties -- the latter case includes also logics that are not even protoalgebraic.

The condition for an algebraic semantics to fall into this class is that the adjoints of $\Xi_A$ of the $i_A$ in the corresponding filter pair also form a natural transformation.
\end{Ex}

The example of congruence filter pairs sheds light on several aspects of the definition of filter pairs:

\begin{enumerate}
    \item The Theorem cited in the above example shows that injectivity of the $i_A$ is a meaningful additional information, and should not be demanded in the axioms of a filter pair. 
    
    \item The condition for the Craig interpolation result, that the family of maps $\Xi_A$ be a natural transformation, points to the usefulness of naturality conditions in logic. Another well-known example is the fact that a protoalgebraic logic is equivalential iff the Leibniz operator is natural. This may serve as a further motivation for the naturality condition in the very definition of filter pair.
    
    \item One can show that if a filter pair gives an equivalent algebraic semantics, the functor $G$ is obtained as the left Kan extension (a universal construction from category theory) of its restriction to the reduced algebras. This gives a new and precise sense in which an algebraizable logic is completely determined by its reduced algebras.
    
    The restriction of the filter pair to reduced algebras is an example of a $\Cc$-filter pair where is $\Cc$ is not all of $\SigmaStr$. This shows that the flexibility of choosing $\Cc$ is also meaningful for semantic considerations.
\end{enumerate}

We can not only express and develop abstract algebraic logic in the language of filter pairs: Other types of filter pairs correspond to other semantics for logics, and one can neatly transfer properly formulated notions and proofs from known cases to new realms:

\begin{Ex}\label{ExampleHornFilterPairs}
Let $\Sigma$ be a first order signature, and $\Tb$ a universal Horn theory. For a $\Tb$-model $A$ define $G^{\Tb}(A)$ to be the set of isomorphism classes of $\Tb$-models receiving a surjective homomorphism from $A$. It becomes a lattice by declaring $B \leq C$ iff there exists a surjective morphism $B \twoheadrightarrow C$ commuting with the given morphisms from $A$. The association $A \mapsto G^{\Tb}(A)$ is part of a contravariant functor $G^{\Tb}$ from $\Tb$-models to lattices. A filter pair whose functor is of the form $G^{\Tb}$ is called \emph{Horn filter pair}.

Just as an algebraic semantics of a logic in a quasivariety is the same thing as a faithful translation of that logic into the equational logic of the quasivariety, a semantics in $\Tb$-models is a faithful translation into the lattice of Horn theories extending $\Tb$, and an equivalent semantics is a translation which has an inverse up to logical equivalence.

In \cite{AMP4} we show :

\begin{center}
\begin{enumerate}
    \item Giving a semantics in $\Tb$-models of a logic is equivalent to giving a presentation of that logic by a Horn filter pair $(G^{\Tb}, i)$.
    \item Giving an  \emph{equivalent}  semantics in $\Tb$-models of a logic is equivalent to giving a presentation of that logic by a Horn filter pair $(G^{\Tb}, i)$, \emph{where the maps $i_A$ are injective}.
\end{enumerate}
\end{center}

Examples: 
\begin{itemize}
    \item If the first order signature $\Sigma$ contains a single unary relation symbol, apart from the function symbols, this formalism captures the generalized matrix semantics that every logic possesses.
    \item  If $\Sigma$ contains a single binary relation symbol $\leq$, apart from the function symbols, and the theory $\Tb$ states that $\leq$ is an order relation, then we obtain Raftery's order algebraic semantics and, in the case where $i$ is injective, the order algebraizable logics.
\end{itemize}

The filter pair proof that under a certain technical condition amalgamation implies Craig interpolation -- mentioned for algebraic semantics in Example \ref{ExampleCongruenceFilterPairs} -- carries over to the setting of Horn filter pairs.
\end{Ex}

As highlighted by the parallel examples \ref{ExampleCongruenceFilterPairs} and \ref{ExampleHornFilterPairs}, filter pairs provide a setting in which one can modularize certain elements of formal semantics of logics and substitute them with other choices, while maintaining the validity of some proofs and interrelations between notions.  This is our main interest in filter pairs, and it will be pursued in \cite{AMP2}, \cite{AMP3}, \cite{AMP4}.

Further topics one can explore in this direction, are Avron's non-deterministic semantics \cite{AvronZamansky}, Ellerman's partition logic \cite{Ellerman} and, more generally, semantics in subobject and quotient lattices of locally presentable categories.
See also \cite{AMP} for some further motivation for filter pairs.

\medskip

The present article, however, does not pursue this semantic direction, but is of a foundational nature, exploring how the abstract structure of a filter pair relates to the logic it presents. In the upcoming sections we first explore, how one can recognize whether a filter pair presents a logic of a given cardinality -- this leads to the $\kappa$-filter pairs of the title. Then we return to the basic intuition of the beginning of this section, that a filter pair is a kind of presentation of a logic, and pursue the question what different presentations by filter pairs a logic admits, obtaining an answer in terms of natural extensions of that logic. 

\section{$\kappa$-Filter pairs}\label{SectionFilterFunctors}

In this section we introduce the notion of $\kappa$-filter pair, discuss some basic properties and show how a $\kappa$-filter pair gives rise to a $\kappa$-logic (see Def. \ref{DefKappaLogic}) and how a logic of cardinality $\kappa$ gives rise to a $reg(\kappa)$-filter pair (where $reg(\kappa)$ denotes the regularization, Def. \ref{DefRegularizationOfCardinal}). 


{\bf From now on the cardinal $\kappa$ is assumed to be \emph{regular},
unless explicitly mentioned otherwise.}

\begin{Df}\label{DefinitionKappaFilterPair}
Let $\Sigma$ be a signature. A \emph{$\kappa$-filter pair} is a pair $(G,i)$ where \linebreak $G\colon\SigmaStr\,^{op}\to \Lk$ is a contravariant functor from the category of $\Sigma$-structures to the category of $\kappa$-presentable lattices and $i = (i_M)_{M\in\SigmaStr}$ is a collection of order preserving functions $i_{M}:G(M)\to (\wp(M);\subseteq)$ with the following properties:

{\bf 1.} For any $A\in\SigmaStr$, $i_{A}$ preserves arbitrary infima (in particular  $i_{A}(\top)=A$) and $\kappa$-directed suprema.

{\bf 2.} Given a morphism $h:M\to N$ the following diagram commutes:

\[
\xymatrix{
M\ar[d]_{h}&G(M)\ar[r]^-{i_{M}}&\wp(M)\\
N&G(N)\ar[u]^{G(h)}\ar[r]_-{i_{N}}&\wp(N)\ar[u]_{h^{-1}}
}
\]
\end{Df}

\begin{Obs}
 Condition 2. says that $i$ is a natural transformation from $G$ to the functor $\wp\colon \SigmaStr^{op}\to \Lk$ sending a $\Sigma$-structure to the power set of its underlying set and a homomorphism of $\Sigma$-structures to its associated inverse image function.
\end{Obs}

The first class of examples of $\kappa$-filter pairs will be established in 
Theorem \ref{FilterPairsFromLogics}, towards which we will work in items \ref{intersectionOfFilters} to \ref{PreimagesOfFiltersAreFilters}.

But first we explain how one can think of a $\kappa$-filter pair as a presentation of a logic of cardinality $\leq \kappa$. 

\begin{Obs}\label{AlgebraicLatticesArePresentableCats}
 A $\kappa$-presentable lattice is equivalent to a small locally $\kappa$-presentable category. Condition 1. then says that each $i_M$, seen as a functor, is accessible and preserves limits. By \cite[Thm. 1.66]{AR} it has a left adjoint, i.e. it is part of a (covariant) Galois connection. We thus get a closure operator on $\wp(M)$ (corresponding to the unit of the adjunction) and a kernel operator (or coclosure operator) on $G(M)$ (corresponding to the counit). We will prove below that the closure operator on $\Fm{\Sigma}{X}$, the absolutely free $\Sigma$-structure over a set $X$, has cardinality $\leq \kappa$ and is structural and hence gives rise to a $\kappa$-logic. This will be the logic associated to the filter pair. We will now spell this out in less category theoretical terms.

\end{Obs}

Recall that an order preserving function $f \colon P \to Q$ between posets is right adjoint to a function $g \colon Q \to P$ (and $g$ is left adjoint to $f$) if the following relation holds for all $p \in P,\ q \in Q$: $g(q) \leq p \ \Leftrightarrow \ q \leq f(p)$, i.e. if $f$ and $g$ form a (covariant) Galois connection. In this case the composition $f \circ g$ is a closure operator on $Q$ and $g \circ f$ a coclosure operator (or kernel operator) on $P$ and $f$, $g$ restrict to a bijection between the (co)closed elements. The (co)closed elements are exactly those elements in the image of $f$, resp. $g$, since from the adjunction properties it follows that $f \circ g \circ f = f$ and $g \circ f \circ g = g$.

It is easy to see that any  $f \colon P \to Q$ that has a (automatically unique) left adjoint, preserves all the  infima existing in $P$. Moreover:

\begin{Teo}\label{LeftAdjointInFilterPairs}
{\em \cite[Thm. 3.6.9]{TaylorPracticalFoundations}} Let $f \colon P \to Q$ be a function between complete posets that preserves arbitrary infima. Then $f$ has a left adjoint $g \colon Q \to P$, given by $g(q):=inf\{p \in P \, \mid \, q \leq f(p) \}$. 
\end{Teo}

Of course, there is a dual result concerning increasing functions that  have a right adjoint or preserve suprema.


%
%
%

By the theorem above, the maps $i_M$ forming the natural transformation of a $\kappa$-filter pair $(G, i)$ have left adjoints $j_M$ (since they preserve arbitrary infima). From this we have the closure operator $i_M \circ j_M$ on each $\Sigma$-structure $M$. In particular for a set $X$ there is a closure operator on $\Fm{\Sigma}{X}$. This defines a logic:

\begin{Prop}\label{PropLogicsFromFilterPairs}
Let $(G,i)$ be a $\kappa$-filter pair and $X$ be a set. For the $\Sigma$-structure $\Fm{\Sigma}{X}$ let $j_\Fm{\Sigma}{X}$ be the left adjoint to $i_\Fm{\Sigma}{X}$. Then the closure operator $C_G:=i_\Fm{\Sigma}{X} \circ j_\Fm{\Sigma}{X}$ defines a logic of cardinality at most $\kappa$ ($\kappa$-logic) on $\Fm{\Sigma}{X}$.
\end{Prop}
\begin{prooof}
By the axioms of a $\kappa$-filter pair, $i_{\Fm{\Sigma}{X}}$ preserves $\kappa$-directed suprema. Since $j_{\Fm{\Sigma}{X}}$ is a left adjoint it preserves arbitrary suprema. Hence the closure operator $C_{G}:=i_{\Fm{\Sigma}{X}} \circ j_{\Fm{\Sigma}{X}}$ preserves $\kappa$-directed suprema. Since any set $S \in \wp(\Fm{\Sigma}{X})$ is $\kappa$-directed union of its subsets of cardinality smaller than $\kappa$, we have that $C_{G}(S)=\bigcup_{S' \subseteq S, |S'| < \kappa} C_{G}(S')$.

It remains to show structurality. Let $\sigma \in hom(\Fm{\Sigma}{X},\Fm{\Sigma}{X})$ and $\Gamma\cup\{\varphi\}\subseteq \Fm{\Sigma}{X}$ such that $\varphi \in C_{G}(\Gamma)$ (i.e. $\Gamma\vdash_{G}\varphi$ in the associated consequence relation). Then we need to show $\sigma(\varphi) \in C_{G}(\sigma(\Gamma))$.

We have $\sigma(\Gamma) \subseteq C(\sigma(\Gamma))=i(j(\sigma(\Gamma)))$ and therefore $\Gamma \subseteq \sigma^{-1}i(j(\sigma(\Gamma)))$. Since the naturality square 

$$\xymatrix{
\Fm{\Sigma}{X} \ar[d]_\sigma &  G(\Fm{\Sigma}{X}) \ar[r]^-i & \wp(\Fm{\Sigma}{X}) \\
\Fm{\Sigma}{X} & G(\Fm{\Sigma}{X}) \ar[u]^{G(\sigma)} \ar[r]^-i & \wp(\Fm{\Sigma}{X}) \ar[u]_{\sigma^{-1}} 
}$$
commutes, we have $\sigma^{-1}(i(j(\sigma(\Gamma)))) = i(G(\sigma)(j(\sigma(\Gamma))))$, so $\sigma^{-1}(i(j(\sigma(\Gamma))))$ is in the image if $i$ and therefore closed.
Hence applying the closure operator $C_{G}$ to the inclusion $\Gamma \subseteq \sigma^{-1}i(j(\sigma(\Gamma)))$ yields $\varphi \in C_{G}(\Gamma) \subseteq C_{G}(\sigma^{-1}i(j(\sigma(\Gamma)))) = \sigma^{-1}i(j(\sigma(\Gamma)))$. Now applying $\sigma$ yields $\sigma(\varphi) \in i(j(\sigma(\Gamma))) = C_{G}(\sigma(\Gamma))$.
\end{prooof}


\begin{Df}\label{DefinitionLogX}
For a filter pair ${\cal F}=(G,i)$ and a set $X$ we will denote the logic obtained from Prop. \ref{PropLogicsFromFilterPairs} by $\Log{X}({\cal F})$.
\end{Df}

%

\begin{Obs}
More generally, and with the same proof, for every $\Sigma$-structure $A$ one obtains an \emph{abstract logic} in the sense of \cite{BloomBrownSuszko}, given by the closure operator $i_A \circ j_A$.

A different description of the consequence relation of this abstract logic is \[D \vdash_A a\ \text{iff  for every}\ z\in G(A), \text{if}\ D\subseteq i_{A}(z)\ \text{then}\ a\in i_{A}(z).\]
The proof is the same as that of \cite[Prop. 2.4]{AMP}. 
\end{Obs}

\medskip

We now show that every $\kappa$-logic 
comes from a  $\kappa$-filter pair (whenever $\kappa$ is a regular cardinal). Let $l = (\Sigma, X, {\vdash})$ be a $\kappa$-logic.

For the following, recall that $\Fi{l}{A}$ denotes the collection of all filters on $A$ (Def. \ref{DefTheoryFilterLatticeOfFilters}).
Lemmas \ref{intersectionOfFilters}, \ref{InclusionOfFiltersPreservesDirectedSuprema} and \ref{PreimagesOfFiltersAreFilters} are well-known but we give short proofs for the convenience of the reader.

\medskip



\begin{Lem}\label{intersectionOfFilters}
An arbitrary intersection of filters is a filter again. In particular $\Fi{l}{A}$ is a complete lattice.
\end{Lem}
\begin{prooof} That filters are closed under intersection is immediate from the definition. Thus the subset $\Fi{l}{A} \subseteq \wp(A)$ has arbitrary infima and hence is a complete lattice. \end{prooof}

\begin{Lem}\label{InclusionOfFiltersPreservesDirectedSuprema}
The inclusion $i_A \colon \Fi{l}{A} \hookrightarrow \wp(A)$ preserves arbitrary infima and $\kappa$-directed suprema.
\end{Lem}
\begin{prooof} The statement about infima is Lemma \ref{intersectionOfFilters}. For the statement about \linebreak suprema we need to show that a $\kappa$-directed union of filters is a filter.

Let $(F_i)_{i\in I}$ be a $\kappa$-directed system of filters. Let 
$\Gamma \cup \{\varphi\} \subseteq \Fm{\Sigma}{X}$ such that $\Gamma \vdash_l \varphi$ and $v \colon \Fm{\Sigma}{X} \to A$ be  a morphism satisfying $v(\Gamma) \subseteq \bigcup_{i \in I}F_i$. 
Since $l$ is of cardinality $ \leq\kappa$, there is $\Gamma' \subseteq \Gamma$ with $|\Gamma'| < \kappa$ such that $\Gamma' \vdash \varphi$. Every element $\gamma \in \Gamma'$ is in some $F_{\gamma}$ and all these $F_{\gamma}$ are contained in some $F_j$, since the system is $\kappa$-directed. Since $F_j$ is a filter, we have that $v(\varphi) \in F_j \subseteq \bigcup_{i \in I}F_i$. This shows the claim.
\end{prooof}

\begin{Lem}\label{FiltersArePresentableLattice}
Let $A$ be a $\Sigma$-structure. Then $\Fi{l}{A}$ is a $\kappa$-presentable lattice.
\end{Lem}
\begin{prooof}
Completeness has been stated in Lemma \ref{intersectionOfFilters}. In particular, given an arbitrary subset $S \subseteq A$ one can form the filter generated by $S$ by setting $\overline{S} := \bigcap_{F \in \Fi{l}{A}, \ S \subseteq F} F$. The operation $\overline{(-)}$ is evidently a closure operation on $\wp(A)$ (indeed it is the closure operation coming from the adjunction of the filter pair).

It remains to show that $\Fi{l}{A}$ is $\kappa$-presentable, i.e. that every $F \in \Fi{l}{A}$ is a $\kappa$-directed supremum of $\kappa$-small elements. 

The filters $\bar{S}$ generated by subsets $S \subseteq A$ with $|S| < \kappa$ are $\kappa$-small elements: Indeed, if $\bar{S} \subseteq \bigvee_{i \in I}F_i$ for some $\kappa$-directed system $(F_i)_{i \in I}$, then also $S \subseteq \bigvee_{i \in I}F_i = \bigcup_{i \in I}F_i$ (the latter equality follows from Lemma \ref{InclusionOfFiltersPreservesDirectedSuprema}, and is explicitly shown in the proof there). Hence each of the less than $\kappa$ many elements of $S$ is in some $F_i$, hence all are simultaneously in some $F_j$ (because $(F_i)_{i \in I}$ is a $\kappa$-directed system), i.e. $S \subseteq F_j$, hence $\bar{S} \subseteq \overline{F_j} = F_j$, the latter equality holding because $F_j$ is a filter.

Now we claim that every $F \in \Fi{l}{A}$ can be written as 
$$F = \bigvee_{F' \subseteq F, \ |F'| < \kappa}\!\!\!\!\!\!\!\!\!\! \overline{F}' = \bigcup_{F' \subseteq F, \ |F'| < \kappa} \!\!\!\!\!\!\!\!\!\! \overline{F}'.$$ 
Indeed, since for every element $f \in F$ the singleton subset $\{f\}$ occurs in the index of the supremum, the inclusion $\subseteq$ holds. On the other hand for every $F'$ occurring in the index of the supremum we have $\overline{F}' \subseteq \overline{F} = F$, hence the inclusion $\supseteq$ holds.
\end{prooof}


\begin{Lem}\label{PreimagesOfFiltersAreFilters}
Preimages of filters under homomorphisms of $\Sigma$-structures are filters again.
\end{Lem}
\begin{prooof}
Let $f \colon A' \to A$ be homomorphism of $\Sigma$-structures and $F \subseteq A$ a filter. To see that $f^{-1}(F) \subseteq A'$ is a filter again, consider $\Gamma \cup \{\varphi\} \subseteq \Fm{\Sigma}{X}$, and a homomorphism $v \colon \Fm{\Sigma}{X} \to A'$ such that $v(\Gamma) \subseteq f^{-1}(F)$. Then $(f \circ v)(\Gamma)=f(v(\Gamma)) \subseteq F$, hence, since $F$ is a filter and $f \circ v$ a homomorphism, $f(v(\varphi)) \in F$, so $v(\varphi) \in f^{-1}(F)$.
\end{prooof}

\medskip

Denote by $i$ the collection of the inclusions $i_A \colon Fi_l(A) \hookrightarrow \wp(A)$.

\begin{Teo}
\label{FilterPairsFromLogics}
 Let $l$ be a $\kappa$-logic. Then $(Fi_l(-), i)$ is a $\kappa$-filter pair.
\end{Teo}
\begin{prooof}
By Lemmas \ref{FiltersArePresentableLattice} and \ref{PreimagesOfFiltersAreFilters} $Fi_l$ is a well defined functor from $\Sigma$-structures to $\kappa$-presentable lattices. It is clear that $i$ is a natural transformation. The remaining condition for a $\kappa$-filter pair is ensured by Lemma \ref{InclusionOfFiltersPreservesDirectedSuprema}.
\end{prooof}

\begin{Df}
We denote the filter pair of Proposition \ref{FilterPairsFromLogics} by \linebreak $\FilPa(l) := (Fi_l(-), i)$ and call it the \emph{canonical filter pair} of the logic $l$. 
\end{Df}

The next theorem says that passing from a logic to a filter pair as in Prop. \ref{FilterPairsFromLogics} and then back to a logic as in Prop. \ref{PropLogicsFromFilterPairs} gives back the same logic.

\begin{Teo}\label{TheoremEveryLogicComesFromAFilterPair}
 Let $l = (\Fm{\Sigma}{X}, {\vdash})$ be a logic. Then the closure operator $i_\Fm{\Sigma}{X} \circ j_\Fm{\Sigma}{X}$ on $\Fm{\Sigma}{X}$ coming from the filter pair $\FilPa(l)$ is equal to the closure operator associated to the consequence relation $\vdash$. In other words, $\Log{X}(\FilPa(l))=l$.
\end{Teo}
\begin{prooof}
The closure operator on $\wp(A)$, for a $\Sigma$-structure $A$, associated to the filter pair $(Fi_l, i)$ is exactly the operator $\bar{(-)}$ from the proof of Lemma \ref{FiltersArePresentableLattice}, which sends a set to the smallest filter containing it. This is true in particular for $A=\Fm{\Sigma}{X}$. The closure operator on $\Fm{\Sigma}{X}$ associated to the consequence relation $\vdash$ is the operator which sends a set to the smallest theory containing it. It thus suffices to show that the filters on the algebra $\Fm{\Sigma}{X}$ are exactly the theories of the logic $l$.

Let $F \subseteq \Fm{\Sigma}{X}$ be a filter for $l$. Let $\Gamma \cup \{\varphi\} \subseteq F$ such that $\Gamma \vdash \varphi$. Then $\Gamma = \id(\Gamma)$, so by the filter property $\varphi = \id(\varphi) \in F$.

On the other hand let $T \subseteq \Fm{\Sigma}{X}$ be a theory. Let $\Gamma \cup \{\varphi\} \subseteq F$ such that $\Gamma \vdash \varphi$ and let $\sigma \colon \Fm{\Sigma}{X} \to \Fm{\Sigma}{X}$ be a homomorphism such that $\sigma(\Gamma) \subseteq T$. Then from substitution invariance we get $\sigma(\Gamma) \vdash \sigma(\varphi)$ and hence, since $T$ is a theory, $\sigma(\varphi) \in T$.
\end{prooof}

\medskip

For the following statement we depart from our standing assumption that the cardinal $\kappa$ is regular.

\begin{Teo} \label{kappa-filter-logic}
The canonical filter pair $\FilPa(l)$ of a logic $l = (Fm_\Sigma(X), {\vdash})$ of cardinality $ \leq \kappa$ -- \emph{where $\kappa$ is allowed to be singular} -- is a $reg(\kappa)$-filter pair.
\end{Teo}

\begin{prooof}
By hypothesis $l$ is a $\k$-logic, thus it is also a $reg(\k)$-logic. Now apply Theorem \ref{FilterPairsFromLogics}.
\end{prooof}

%
%
\section{Natural Extensions}\label{SectionNaturalExtensions}

\begin{Df}
For sets $X \subseteq Y$ of variables, a {\bf natural extension} of a logic $l=(\Fm{\Sigma}{X}, {\vdash}_l)$ to $\Fm{\Sigma}{Y}$ is a logic $(\Fm{\Sigma}{Y}, {\vdash})$ which is a conservative extension of $l$ with the same cardinality as $l$.
\end{Df}

One reason for studying natural extensions in Abstract Algebraic Logic is that some proofs of transfer theorems, that are central in it, require the existence of extensions of logics to bigger sets of variables.

\medskip

We begin this section by listing four tentative constructions of natural extensions and summarizing the results on them and their interrelations. 
In the context of constructing a natural extension of a logic $l$ of cardinality $\kappa$ -- where $\kappa$ is a regular cardinal -- the following relations between subsets and elements of $\Fm{\Sigma}{Y}$ have been defined in the literature:

\begin{enumerate}[(a)]
 \item ${\vdash}^{\textrm{\tiny {\L}S}}$ ({\L}o\'s-Suszko), defined by $$\Gamma \vdash^{\textrm{\tiny {\L}S}} \varphi\ \ \ i\!f\!\!f \ \ there\ are\ an\ automorphism\ v \colon \Fm{\Sigma}{Y} \to \Fm{\Sigma}{Y}$$
 $$and\ \Gamma' \subseteq \Gamma\ \ and\ \varphi \ s.t.\ \ v(\Gamma' \cup \{\varphi\})\subseteq \Fm{\Sigma}{X}\ and\ v(\Gamma')\vdash_l v(\varphi).$$

 \item ${\vdash}^{\textrm{\tiny SS}}$ (Shoesmith-Smiley), defined by $$\Gamma \vdash^{\textrm{\tiny SS}} \varphi\ \ \ i\!f\!\!f \ \ there\ are\ \Gamma' \cup \varphi' \subseteq \Fm{\Sigma}{X}\ and\ v \colon X \to \Fm{\Sigma}{Y}\ s.t.$$ $$ v(\Gamma') \subseteq \Gamma, v(\varphi')=\varphi\ and\ \Gamma' \vdash_{l_X} \varphi'.$$

 \item ${\vdash}^{-}$ (P{\v r}enosil), the smallest consequence relation on $\Fm{\Sigma}{Y}$ satisfying the rules $\Gamma \vdash^{-} \varphi$ whenever $\Gamma \cup \{\varphi\} \subseteq \Fm{\Sigma}{X}$ and $\Gamma \vdash_l \varphi$.
  
 \item ${\vdash}^{+}_\kappa$ (P{\v r}enosil), defined as the $\kappa$-ary part (see Def. \ref{DefKappaAryPartOfALogic}) of the relation ${\vdash}^{+}$, given by 
 $$\Gamma \vdash^{+} \varphi \ \ \ i\!f\!\!f \ \ \sigma(\Gamma)\vdash_l\sigma(\varphi)\ for\ every\ substitution\ \sigma \colon \Fm{\Sigma}{Y} \to \Fm{\Sigma}{X}.$$
 \end{enumerate}

The {\L}o\'s-Suszko relation ${\vdash}^{\textrm{\tiny {\L}S}}$ is a conservative extension of ${\vdash}_l$ to $\Fm{\Sigma}{Y}$ which satisfies monotonicity and reflexivity \cite[Prop. 16]{Prenosil} and is clearly $\kappa$-ary, but may fail to satisfy structurality  \cite[Prop. 18]{Prenosil}.

While the Shoesmith-Smiley relation ${\vdash}^{\textrm{\tiny SS}}$ was for a while thought to always yield a natural extension, this was shown not to be the case in general by Cintula and Noguera. It is always a conservative extension of ${\vdash}_l$ that satisfies monotonicity, structurality, reflexivity and is $\kappa$-ary \cite[Lem. 2.4]{CintulaNoguera}, but it may fail to satisfy the Cut rule (i.e. idempotence) \cite[Prop. 2.8]{CintulaNoguera}.
 
The {\L}o\'s-Suszko relation is always contained in the Shoesmith-Smiley relation and they coincide if either $|X|<|Y|$ or $\card(l) \leq |X|$ \cite[Lem. 2.7]{CintulaNoguera} \cite[Prop. 15]{Prenosil} or if the {\L}o\'s-Suszko relation actually is a logic \cite[Thm. 17]{Prenosil}. Since structurality can fail for ${\vdash}^{\textrm{\tiny {\L}S}}$ but not for ${\vdash}^{\textrm{\tiny SS}}$, they need not coincide in general.

In view of their results Cintula and Noguera asked whether a logic always has a natural extension to a given bigger set of variables. For logics of regular cardinality $\kappa$, P{\v r}enosil gave an affirmative answer: both ${\vdash}^{-}$ and ${\vdash}^{+}_\kappa$ are always natural extensions of $l$, with ${\vdash}^-$ being the minimal and the ${\vdash}^{+}_\kappa$ the maximal one \cite[Prop. 7, Cor. 6]{Prenosil}.

Furthermore, Cintula and Noguera showed that whenever $|X|=|Y|$ or \linebreak $\card(l_X)\leq |X|^+$ 
there is a \emph{unique} natural extension and that it is given by the Shoesmith-Smiley relation \cite[Thm 2.6]{CintulaNoguera}. They asked whether there is always a unique natural extension, which P{\v r}enosil showed not to be the case \cite[Prop. 19, Prop. 20]{Prenosil}.

\begin{Obs}\label{ObsPorqueTratamosDoCasoRegular} We now elucidate the assumption of the regularity of $\kappa$.
 In constructing the minimal and the maximal natural extensions 
 $L^-:=((\Fm{\Sigma}{Y}, {\vdash}^-))$ and $L^+_{\kappa}:=(\Fm{\Sigma}{Y}, {\vdash}_\kappa^+)$, P{\v r}enosil leaves implicit the assumption that the cardinality of the logic in question is a regular cardinal. 

The logic $L^+_{\kappa}$ is explicitly defined by taking the $\kappa$-ary part of another logic, and thus by Fact \ref{FactForRegularCardinalKappaAryPartIsLogic} exists for {regular} $\kappa$, but it is not clear whether it exists in general for singular $\kappa$, see Example \ref{ExampleKappaAryPartCanFailToBeLogic}. 

The \emph{construction} of the logic $L^-$ does not involve taking the $\kappa$-ary part of another logic, but the proof that the resulting logic is $\kappa$-ary (the proof of \cite[Prop. 7]{Prenosil}) uses the $\kappa$-ary part and a priori again only works for regular $\kappa$.

Further, in \cite[Cor. 8]{Prenosil}, P{\v r}enosil characterizes $L^-$ as the logic over the language with the enhanced set of variables generated by the rules of the original logic. This is another construction that does not involve taking the $\kappa$-ary part of a logic, but a closure process like generating a logic from rules is the kind of thing where one often passes from a cardinal to its regularization (see Def. \ref{DefRegularizationOfCardinal}) as it happens e.g. in Example \ref{ExampleKappaAryPartCanFailToBeLogic}. So it is not clear that this offers a way around the regularity assumption.

As it stands, it thus remains an open problem whether every logic of singular cardinality has a natural extension. What we show below about the singular case, is the next best thing, namely that P{\v r}enosil's construction gives a conservative extension of cardinality at most the successor of the cardinality of the original logic.

The existence of natural extensions in this remaining open case is, however, a problem of no practical importance. Singular cardinals are rare (the smallest one is $\aleph_{\omega}$) and logics of singular cardinality are to our knowledge unheard of in concrete applications. 

We merely wish to point out this state of affairs, in order to explain the appearance of the regularity assumptions in this work.
\end{Obs}

\medskip

{\bf In the following we keep the standing assumption that all occurring logics have regular cardinality.}

\medskip

We start by shedding some more light on the connection between the {\L}o\'s-Suszko relation and the Shoesmith-Smiley relation. As we just remarked, both these relations are monotonous and reflexive and the former can fail to be structural, while the latter is always structural. Since relations that are monotonous, reflexive and structural are closed under arbitrary intersections, there is a smallest such relation containing ${\vdash}^{\textrm{\tiny {\L}S}}$, which we call its structural closure.

\begin{Prop}\label{PropShoesmithSmileyIsStructuralClosureOfLosSuszko}
 The Shoesmith-Smiley relation ${\vdash}^{\textrm{\tiny SS}}$ is the structural closure of the {\L}o\'s-Suszko relation  ${\vdash}^{\textrm{\tiny {\L}S}}$.
\end{Prop}
\begin{prooof}
Denote by ${\vdash}$ the structural closure of the {\L}o\'s-Suszko relation, i.e. the intersection of all monotonous, reflexive and structural relations containing ${\vdash}^{\textrm{\tiny {\L}S}}$. Since by the above remarks the Shoesmith-Smiley relation ${\vdash}^{\textrm{\tiny SS}}$ occurs in this intersection, we have ${\vdash}\ \subseteq\ {\vdash}^{\textrm{\tiny SS}}$.

\medskip

For the opposite inclusion note that by taking the inverse of the automorphism $v$ in the definition of the {\L}o\'s-Suszko relation, one arrives at the description
$$\Gamma \vdash^{\textrm{\tiny {\L}S}} \varphi\ \ \ i\!f\!\!f \ \ there\ are\ an\ automorphism\ v \colon \Fm{\Sigma}{Y} \to \Fm{\Sigma}{Y}$$
 $$and\ \Gamma' \cup \{\varphi'\}\subseteq \Fm{\Sigma}{X} \ \ s.t.\ \ \Gamma' \vdash_l \varphi',\ \ v(\varphi')=\varphi\ \ and\ \  v(\Gamma') \subseteq \Gamma.$$
This says that the pairs $(\Gamma, \varphi)$ with $\Gamma \vdash^{\textrm{\tiny {\L}S}} \varphi$ are exactly the images under $\Fm{\Sigma}{Y}$-automorphisms of pairs $(\Gamma', \varphi')$ with $\Gamma' \vdash_l \varphi'$. The structural closure ${\vdash}$ the contains all images under $\Fm{\Sigma}{Y}$-\emph{endo}morphisms of pairs $(\Gamma', \varphi')$ with $\Gamma' \vdash_l \varphi'$. But this says exactly that ${\vdash}^{\textrm{\tiny SS}} \subseteq {\vdash}$.
\end{prooof}

\medskip

It follows that if ${\vdash}^{\textrm{\tiny {\L}S}}$ is already structurally closed, it coincides with ${\vdash}^{\textrm{\tiny SS}}$. In particular this implies P{\v r}enosil's result that if the {\L}o\'s-Suszko relation is already a logic, then so is the Shoesmith-Smiley relation and the two coincide \cite[Thm. 17]{Prenosil}.

\medskip

As stated, the question of whether there always exists a natural extension has been answered by P{\v r}enosil, with his two constructions. Next we show, how natural extensions are also easily obtained through the language of filter pairs. We show that these natural extensions coincide with P{\v r}enosil's minimal ones and complete the picture by relating the {\L}o\'s-Suszko and Shoesmith-Smiley relations to this one.

\begin{Teo}\label{TheoremFilterPairsGiveConservativeExtensions}
Let $\Sigma$ be a signature, $(G,i)$ a filter pair over $\Sigma$ and $X,Y$ sets with $X \subseteq Y$. Then the induced inclusion $\Fm{\Sigma}{X} \to \Fm{\Sigma}{Y}$ is a conservative translation  $\Log{X}(G,i) \to \Log{Y}(G,i)$.
\end{Teo}
\begin{prooof}
 Denote the inclusion by $\sigma \colon \Fm{\Sigma}{X} \to \Fm{\Sigma}{Y}$. 
 { Choose a map  $\tilde{\tau} \colon {Y} \to \Fm{\Sigma}{X}$ such that $\tilde{\tau}|_{{X}}=\id_{{X}}$,}\footnote{Remember that our sets of variables are infinite, in particular $X \neq \emptyset$.}
 thus the induced homomorphism
  $\tau  \colon \Fm{\Sigma}{Y} \to \Fm{\Sigma}{X}$ is a left inverse of $\sigma$, i.e. $\tau \circ \sigma = \textrm{id}_{\Fm{\Sigma}{X}}$.
  We then have the following diagram (which is commutative if one deletes the $j_X, j_Y$):

$$\xymatrix{
\Fm{\Sigma}{X} \ar@{^(->}[d]_\sigma & G(\Fm{\Sigma}{X}) \ar@<.5ex>[r]^{i_X} & \wp(\Fm{\Sigma}{X}) \ar@<.5ex>[l]^{j_X} \\
\Fm{\Sigma}{Y} \ar@{->>}[d]_\tau & G(\Fm{\Sigma}{Y}) \ar@<.5ex>[r]^{i_Y} \ar[u]^{G(\sigma)} & \wp(\Fm{\Sigma}{Y})  \ar[u]_{\sigma^{-1}}  \ar@<.5ex>[l]^{j_Y} \\
\Fm{\Sigma}{X} & G(\Fm{\Sigma}{X}) \ar@{=}@/^3pc/[uu]^{\textrm{id}} \ar@<.5ex>[r]^{i_X} \ar[u]^{G(\tau)} & \wp(\Fm{\Sigma}{X}) \ar@<.5ex>[l]^{j_X} \ar[u]_{\tau^{-1}}  \ar@{=}@/_3pc/[uu]_{\textrm{id}}
}$$

Note that $\sigma^{-1}(Z) = Z \cap \Fm{\Sigma}{X}$.

\medskip

Abbreviating $l_X:=\Log{X}(G,i)$ and $l_Y:=\Log{Y}(G,i)$, we need to show that for $\Gamma \cup \{\varphi\} \subseteq \Fm{\Sigma}{X}$ we have
\[\Gamma\vdash_{l_X}\varphi\ \ \text{iff} \ \ \Gamma\vdash_{l_Y}\varphi.\]

\medskip

``$\Rightarrow$'' Suppose that $\Gamma\vdash_{l_X}\varphi$. We need to show $\Gamma\vdash_{l_Y}\varphi$, i.e. $\varphi \in i_Yj_Y(\Gamma)$. Since $\varphi \in \Fm{\Sigma}{X}$, this is equivalent to $$\varphi \in i_Yj_Y(\Gamma) \cap \Fm{\Sigma}{X} = \sigma^{-1} i_Y j_Y(\Gamma) = i_X G(\sigma) j_Y(\Gamma),$$ where the last equality holds because of naturality.

Since $\Gamma \subseteq i_Yj_Y(\Gamma)$, and again since $\Gamma \subseteq \Fm{\Sigma}{X}$, we have $\Gamma \subseteq i_Yj_Y(\Gamma) \cap \Fm{\Sigma}{X} = i_X G(\sigma) j_Y(\Gamma)$.
  Since $\Gamma\vdash_{l_X}\varphi$, every set in the image of $i_X$ that contains $\Gamma$ also contains $\varphi$, so $\varphi \in i_X G(\sigma) j_Y(\Gamma) = i_Yj_Y(\Gamma) \cap \Fm{\Sigma}{X} \subseteq i_Yj_Y(\Gamma)$.
  
  \medskip

``$\Leftarrow$'' 
Suppose that $\Gamma\vdash_{l_Y}\varphi$. We know that $\Gamma \subseteq i_X j_X (\Gamma)$. Since 
$\tau \circ \sigma = \textrm{id}_{\Fm{\Sigma}{X}}$ this implies $\Gamma \subseteq \tau^{-1} i_X j_X (\Gamma) = i_Y G(\tau) j_X (\Gamma)$ (the equality again coming from the naturality square). Since $\Gamma\vdash_{l_Y}\varphi$, every set in the image of $i_Y$ that contains $\Gamma$ also contains $\varphi$. As $\varphi \in \Fm{\Sigma}{X}$, it follows that 
$$\begin{array}{rcl}
 \varphi \in i_Y ( G(\tau) ( j_X (\Gamma) )) \cap \Fm{\Sigma}{X} &=& \sigma^{-1} ( i_Y ( G(\tau) ( j_X (\Gamma)))) \\
 &=& i_X ( G(\sigma) ( G(\tau) ( j_X(\Gamma)))) = i_X ( j_X(\Gamma)).
\end{array}
$$
\end{prooof}

\begin{Cor}\label{CorollaryKappaFilterPairPresentingLogicOfCardinalityKappaAlsoYieldsNaturalExtensions}
 Let $(G,i)$ be a $\kappa$-filter pair, $X$ a set, and suppose that \linebreak $\card\, \Log{X}(G,i) = \kappa$. Then for every $Y \supseteq X$ the logic $\Log{Y}(G,i)$ is a natural extension of $\Log{X}(G,i)$.
\end{Cor}
\begin{prooof}
 We know from Theorem \ref{TheoremFilterPairsGiveConservativeExtensions} that $\Log{Y}(G,i)$ is a conservative extension of $\Log{X}(G,i)$. Since $\Log{Y}(G,i)$ is presented by a $\kappa$-filter pair, we have  $\card\, \Log{Y}(G,i) \leq \kappa$. Finally, since by hypothesis $\card\, \Log{X}(G,i) = \kappa$, for every cardinal $\rho < \kappa$ there are formulas $\Gamma \cup \{\varphi\} \subseteq \Fm{\Sigma}{X}$ such that $\Gamma \vdash_{\Log{X}(G,i)} \varphi$ and for no subset $\Gamma' \subseteq \Gamma$ with $|\Gamma'|<\rho$ one has $\Gamma' \vdash_{\Log{X}(G,i)} \varphi$. As $\Log{Y}(G,i)$ is a conservative extension of $\Log{X}(G,i)$, we also have for no subset $\Gamma' \subseteq \Gamma$ with $|\Gamma'|<\rho$ that $\Gamma' \vdash_{\Log{Y}(G,i)} \varphi$, showing that $\card\, \Log{Y}(G,i) \geq \kappa$, and hence $\card\, \Log{Y}(G,i) = \kappa$.
\end{prooof}

\begin{Cor}[P{\v r}enosil \cite{Prenosil}]
 Let $X,Y$ be sets, $X \subseteq Y$. Then every logic over $\Fm{\Sigma}{X}$ has a natural extension to $\Fm{\Sigma}{Y}$.
\end{Cor}
\begin{prooof}
 We know from Theorem \ref{TheoremEveryLogicComesFromAFilterPair} that every logic of cardinality $\kappa$ can be presented by a $\kappa$-filter pair. Hence the claim follows from Corollary \ref{CorollaryKappaFilterPairPresentingLogicOfCardinalityKappaAlsoYieldsNaturalExtensions}.
\end{prooof}
\medskip

Our results so far, for a logic \emph{singular} cardinality, do not give a natural extension, but the next best thing:

\begin{Cor}\label{CorExtNatCasoSingular}
Let $X,Y$ be sets, $X \subseteq Y$. Then every logic of \emph{singular} cardinality $\kappa$ over $\Fm{\Sigma}{X}$ has a conservative extension to $\Fm{\Sigma}{Y}$ of cardinality at most $\kappa^+$.
\end{Cor}
\begin{prooof}
By Thm. \ref{kappa-filter-logic} the logic can be presented by a $\kappa^+$-filter pair $(G,i)$. By Thm. \ref{TheoremFilterPairsGiveConservativeExtensions} the logic $\Log{Y}(G,i)$ is a conservative extension, and, coming from a $\kappa^+$-filter pair, it has cardinality $<\kappa^+$. 
\end{prooof}

\medskip

We proceed to pin down the precise relationships between the several (tentative) constructions of natural extensions. As Cintula and Noguera proved, the only thing that can fail with Shoesmith-Smiley's tentative definition of a natural extension is idempotence. Next we show, in Proposition \ref{PropOurNaturalExtensionIsIdempotentHullOfShoesmithSmileysRelation} below, that if one takes Shoesmith-Smiley's relation ${\vdash}^{\textrm{\tiny SS}}$ and forces it to be idempotent, one obtains our consequence relation on $\Log{Y}(G,i)$. To show this we review some facts about idempotent hulls.

\begin{construction}
  Consider a set $M$ and an increasing, monotonous operation $E \colon \wp(M) \to \wp(M)$. There is a smallest idempotent operation $C \colon \wp(M) \to \wp(M)$ which is bigger than $E$ in the setwise order, i.e. satisfying $E(X) \subseteq C(X)$ for all $X \in \wp(M)$. One can construct it by iterating the operation $E$ until nothing changes anymore:

For an ordinal number $\alpha$ we define inductively $E^{\alpha+1}(X):=E(E^{\alpha}(X))$ for a successor ordinal, and $E^\alpha(X):=\bigcup_{\beta < \alpha}E^\beta(X)$ for a limit ordinal $\alpha$. Since $E$ is monotonous and increasing, we have $E^\alpha(X) \subseteq E^{\beta}(X)$ whenever $\alpha<\beta$. It is also clear that each $E^\alpha$ is itself monotonous and increasing. Choose a limit ordinal $\gamma$ with $|\gamma|>|M|$. Since $\wp(M)$ does not contain chains of \emph{strict} inclusions indexed by the ordinal $\gamma$, we have $E^\gamma(X)=E^{\gamma+1}(X)$. Now we define the operator $C \colon \wp(M) \to \wp(M)$ by $C(X):=E^\gamma(X) = \bigcup_{\alpha < \gamma} E^\alpha(X)$. We have $E(C(X))=C(X)$, hence $E^\alpha(C(X))=C(X)$ for all ordinals $\alpha$, and hence $C(C(X))= \bigcup_{\alpha < \gamma} E^\alpha(C(X)) = \bigcup_{\alpha < \gamma} C(X) = C(X)$. So $C$ is an increasing, monotonous and idempotent operator containing $E$. Any other such operator needs to contain all iterations of $E$ and hence also $C$, so $C$ is the smallest such operator.

The operator $C$ just constructed is called the \emph{idempotent hull} of $E$.
\end{construction}

\medskip

Call a subset $Y \subseteq M$ \emph{$E$-closed} if, whenever $\varphi \in E(Y)$, one also has $\varphi \in Y$, i.e. if $E(Y)=Y$.

\begin{Lem}\label{LemmaDescriptionIdempotentHull}
Let $X \subseteq M$. Then $C(X)$ is the smallest $E$-closed subset of $M$ containing $X$. 
\end{Lem}
\begin{prooof} $C(X)$ is $E$-closed by the observation $E(C(X))=C(X)$ from above. If $Y$ is another $E$-closed set containing $X$, then %
$C(X)=\bigcup_{\alpha < \gamma}E^\alpha(X)$ $\subseteq \bigcup_{\alpha < \gamma}E^\alpha(Y) = Y,$ where the inclusion comes from the monotonicity of the operators $E^\alpha$. \end{prooof}

\begin{Prop}\label{PropOurNaturalExtensionIsIdempotentHullOfShoesmithSmileysRelation}
Let $E, C \colon \wp(\Fm{\Sigma}{Y}) \to \wp(\Fm{\Sigma}{Y})$ be the operations \linebreak  given by 
$E(\Gamma):=\{\varphi \, \mid \, \Gamma \vdash^{\textrm{\tiny SS}} \varphi \}$ and $C(\Gamma):=\{\varphi \, \mid \, \Gamma \vdash_{\Log{Y}(\FilPa(l))} \varphi \}$, respectively. Then the operation $C$ is the idempotent hull of $E$.
\end{Prop}
\begin{prooof}
By definition we have that $\varphi \in E(\Gamma)$ iff $\exists \Gamma'\cup\{\varphi'\}\subseteq \Fm{\Sigma}{X}$ and $v\colon \Fm{\Sigma}{X} \to \Fm{\Sigma}{Y}$ such that $\Gamma ' \vdash \varphi '$, $v(\Gamma')\subseteq \Gamma$ and $v(\varphi ')=\varphi$.

The operator $C$ on the other hand is the the closure operator of the logic $\Log{Y}(\FilPa(l))$, and thus by definition associates to a set $Z \subseteq \Fm{\Sigma}{Y}$ the smallest $l$-filter containing $Z$.

In other words, by definition of $l$-filter, $\varphi \in C(\Gamma)$ means that $\varphi$ is contained in the smallest set $Z$ of formulas {{ on the variables $Y$}}  that contains $\Gamma$ and that, whenever there are $\Gamma' \cup \{\varphi'\} \subseteq \Fm{\Sigma}{X}$ s.t. $\Gamma' \vdash_{l_X} \varphi'$ and a morphism $v \colon \Fm{\Sigma}{X} \to \Fm{\Sigma}{Y}$ such that $v(\Gamma') \subseteq Z$ then also $v(\varphi') \in Z$. The latter condition is exactly the condition of being $E$-closed, hence the claim follows from Lemma \ref{LemmaDescriptionIdempotentHull}.
\end{prooof}

\begin{Lem}\label{LemmaInclusionShoesmithSmileyMinimalNatExt}
There is an inclusion  $ {\vdash}^{\textrm{\tiny SS}} \ \subseteq \ {\vdash}^-$.
\end{Lem}
\begin{prooof}
Remember the definition of ${\vdash}^{-}$ as the smallest consequence relation on $\Fm{\Sigma}{Y}$ satisfying the rules $\Gamma \vdash^{-} \varphi$ whenever $\Gamma \cup \{\varphi\} \subseteq \Fm{\Sigma}{X}$ and $\Gamma \vdash_l \varphi$.

Let $\Gamma \vdash^{\textrm{\tiny SS}} \varphi$. By definition there are $\Gamma' \cup \varphi' \subseteq \Fm{\Sigma}{X}$ and  $v \colon X \to \Fm{\Sigma}{Y}$ such that $v(\Gamma') \subseteq \Gamma, v(\varphi')=\varphi$ and $\Gamma' \vdash_{l} \varphi'$.

By definition $\Gamma' \vdash_{l} \varphi'$ implies $\Gamma' \vdash^{-} \varphi'$. Choose any extension $\tilde{v}$ of $v$ to all of $Y$. Then $\tilde{v}(\Gamma')=v(\Gamma)\subseteq \Gamma$ and $\tilde{v}(\varphi')={v}(\varphi')=\varphi$ and since ${\vdash}^{-}$ is structural and monotonous, we have $\Gamma \vdash^{-} \varphi$.
\end{prooof}

\medskip

With this we can start tying together all the different relations considered in this section.

\begin{Cor}\label{CorIdempotentHullOfShoesmithSmileyIsMinimalNatExt}
 The idempotent hull of the Shoesmith-Smiley relation is the minimal natural extension ${\vdash}^-$.
\end{Cor}
\begin{prooof}
Apply the idempotent hull construction to both sides of the inclusion of Lemma \ref{LemmaInclusionShoesmithSmileyMinimalNatExt}. Then we obtain an inclusion between consequence relations.

The left hand side becomes a natural extension of the initial logic $l$ by Prop. \ref{PropOurNaturalExtensionIsIdempotentHullOfShoesmithSmileysRelation} (namely the natural extension coming from the canonical filter pair of $l$) and the right hand side does not change. Since the right hand side is the \emph{minimal} natural extension of $l$, we also have the opposite inclusion.
\end{prooof}

\medskip

\begin{Cor}\label{CorOurNatExtensionIsTheMinimalOne}
 The natural extension ${\vdash}_{\Log{Y}(\FilPa(l))}$ of Corollary \ref{CorollaryKappaFilterPairPresentingLogicOfCardinalityKappaAlsoYieldsNaturalExtensions}, obtained from the canonical filter pair of $l$, is the minimal natural extension ${\vdash}^-$.
\end{Cor}
\begin{prooof}
Immediate from Corollary \ref{CorIdempotentHullOfShoesmithSmileyIsMinimalNatExt} and Proposition \ref{PropOurNaturalExtensionIsIdempotentHullOfShoesmithSmileysRelation}.
\end{prooof}

\medskip

We summarize the results so far in the following theorem.

\begin{Teo}\label{TheoremInclusionsBetweenTheRelations}
 Given sets $X \subseteq Y$ of variables and a logic $l=(\Fm{\Sigma}{X}, {\vdash})$ we have the following inclusions between the associated relations
 $${\vdash}^{\textrm{\tiny {\L}S}} \ \ \subseteq \ \ {\vdash}^{\textrm{\tiny SS}} \ \ \subseteq \ \ {\vdash}^-  \ \ = \ \ {\vdash}_{\Log{Y}(\FilPa(l))}  \ \ \subseteq \ \ {\vdash}^+_\kappa, $$
where the second relation is the structural closure of the first one and the third is the idempotent closure of the second one.
 \end{Teo}
\begin{prooof}
The first inclusion has been noted in \cite[Thm. 17]{Prenosil}, the statement about structural closure is Proposition \ref{PropShoesmithSmileyIsStructuralClosureOfLosSuszko}. The second inclusion is Lemma \ref{LemmaInclusionShoesmithSmileyMinimalNatExt} and the statement about the idempotent hull is Corollary \ref{CorIdempotentHullOfShoesmithSmileyIsMinimalNatExt}. The equality is Corollary \ref{CorOurNatExtensionIsTheMinimalOne}. The final inclusion follows from P{\v r}enosil's result that ${\vdash}^-$ is the minimal and ${\vdash}^+_\kappa$ the maximal natural extension. \cite[Prop. 7, Cor. 6]{Prenosil}
\end{prooof}

\medskip

 As stated at the beginning of the section, uniqueness of natural extensions holds only under certain cardinality restrictions. One can deduce this result in the language of filter pairs by directly proving the independence of the notion of filter from the choice of natural extensions. In this, Cintula and Noguera's cardinality conditions show up for the same reasons as they do in their original work. 

\begin{Prop}\label{PropUnderCardinalityConditionsFiltersAreTheSameForNaturalExtensions}
 Let $X,Y$ be sets, $X \subseteq Y$, and $l_X=(\Sigma, X, {\vdash}), l_Y=(\Sigma, Y, {\vdash}')$ logics such that $l_Y$ is a natural extension of $l_X$. Suppose that either $|X|=|Y|$ or $\card(l_X)\leq |X|^+$ 
 . Then a subset $F$ of a $\Sigma$-structure $A$ is an $l_X$-filter iff it is an $l_Y$-filter.
\end{Prop}
\begin{prooof}
Let $A$ be a $\Sigma$-structure and $F \subseteq A$.

An $l_Y$-filter is an $l_X$-filter: Indeed, let $F$ be an $l_Y$-filter, $\Gamma \cup \{\varphi\} \subseteq \Fm{\Sigma}{X}$ such that $\Gamma \vdash_{l_X} \varphi$, and $v \colon \Fm{\Sigma}{X} \to A$ a valuation with $v(\Gamma) \subseteq F$. We need to show that $v(\varphi) \in F$. 

{Choose a map $g \colon Y \to \Fm{\Sigma}{X}$ such that $g(x):=x$ for $x \in X$.}
This induces a homomorphism $\hat{g} \colon \Fm{\Sigma}{Y} \to \Fm{\Sigma}{X}$ and hence a  valuation $(v \circ \hat{g}) \colon  \Fm{\Sigma}{Y} \to A$. We have $(v \circ \hat{g})(\Gamma) = v(\Gamma)  \subseteq F$ and hence, since $F$ is an $l_Y$-filter, $(v \circ \hat{g})(\varphi) \in F$. Since $v \circ \hat{g}$ coincides with $v$ on $\Fm{\Sigma}{X}$ this means $v(\varphi) \in F$.

\medskip

An $l_X$-filter is an $l_Y$-filter: Let $F$ be an $l_X$-filter, $\Gamma \cup \{\varphi\} \subseteq \Fm{\Sigma}{Y}$ such that $\Gamma \vdash_{l_Y} \varphi$, and $v \colon \Fm{\Sigma}{Y} \to A$ a valuation with $v(\Gamma) \subseteq F$. We need to show that $v(\varphi) \in F$.

Choose $\Gamma' \subseteq \Gamma$ with $|\Gamma'| < \card(l_Y)$. Since $\card(l_Y) = \card(l_X) \leq |X|^+$, we have that $|\Gamma'| \leq |X|$ and also  $|\Gamma' \cup \{\varphi'\}| \leq |X|$, since $X$  is infinite. 
Since every formula of $\Gamma' \cup \{\varphi'\}$ only has finitely many variables, we have that the set $Var(\Gamma' \cup \{\varphi'\})$ of variables ocurring there has cardinality $\leq |X|$. Hence we can choose functions $\tau \colon Y \to Y$ and $\sigma \colon Y \to Y$ such that $\tau$ maps $Var(\Gamma' \cup \{\varphi'\})$ injectively to $X$ and $(\sigma \circ \tau)|_{Var(\Gamma' \cup \{\varphi'\})} = id$. As usual we keep the notations $\tau, \sigma$ for the induced maps on the formula algebra.

We then have $\tau(\Gamma) \cup \{\tau(\varphi)\} \subseteq \Fm{\Sigma}{X}$. By substitution invariance of $l_Y$ we have $\tau(\Gamma) \vdash_{l_Y} \tau(\varphi)$ and, since $l_Y$ is a conservative extension of $l_X$, also $\tau(\Gamma) \vdash_{l_X} \tau(\varphi)$.

 Then $w:= v \circ \sigma|_{\Fm{\Sigma}{X}} \colon \Fm{\Sigma}{X} \to \Fm{\Sigma}{Y} \to A$ is a valuation with $w(\tau(\Gamma)) = v(\sigma(\tau(\Gamma))) = v(\Gamma) \subseteq F$. Since $F$ is an $l_X$-filter, we have $v(\varphi) = v(\sigma(\tau(\varphi))) =  w(\tau(\varphi)) \in F$.
\end{prooof}

%

\begin{Cor}[Cintula, Noguera]\label{CorollaryUnderCardinalityConditionsNaturalExtensionsAreUnique}
 Under the cardinality restrictions of \linebreak Proposition \ref{PropUnderCardinalityConditionsFiltersAreTheSameForNaturalExtensions}, natural extensions are unique. 
\end{Cor}
\begin{prooof}
Let $l_X$ be a logic with set of variables $X$ and $l_Y$ a natural extension of $l_X$ with set of variables $Y$. By Proposition \ref{PropUnderCardinalityConditionsFiltersAreTheSameForNaturalExtensions} we have $\Fi{l_X}{A} = \Fi{l_Y}{A}$ for any $\Sigma$-structure $A$ and hence the equality of filter pairs $\FilPa(l_X) = \FilPa(l_Y)$. By Theorem \ref{TheoremEveryLogicComesFromAFilterPair} $l_Y = \Log{Y}(\FilPa(l_Y))$. Therefore $l_Y = \Log{Y}(\FilPa(l_Y)) = \Log{Y}(\FilPa(l_X))$ is the natural extension of Corollary \ref{CorollaryKappaFilterPairPresentingLogicOfCardinalityKappaAlsoYieldsNaturalExtensions}. 
\end{prooof}

\begin{Obs}
We now have a second proof of Theorem \ref{TheoremEveryLogicComesFromAFilterPair}: By Corollary \ref{CorollaryKappaFilterPairPresentingLogicOfCardinalityKappaAlsoYieldsNaturalExtensions}, in the special case $X=Y$ 
 we obtain that $\Log{X}(\FilPa(l))$ is a natural extension of $l$. Of course $l$ is also a natural extension of itself and the cardinality conditions of Corollary \ref{CorollaryUnderCardinalityConditionsNaturalExtensionsAreUnique} are satisfied, so $\Log{X}(\FilPa(l))=l$. This shows that, given Cintula and Noguera's uniqueness result, Corollary \ref{CorollaryKappaFilterPairPresentingLogicOfCardinalityKappaAlsoYieldsNaturalExtensions} is in fact a generalization of Theorem \ref{TheoremEveryLogicComesFromAFilterPair}.
\end{Obs}

\section{Filter pairs yielding a fixed logic}\label{SectionFilterPairsYieldingAFixedLogic}

We have seen in Theorem \ref{TheoremFilterPairsGiveConservativeExtensions} that a $\kappa$-filter pair can be regarded as a presentation of a family of logics over all sets of variables, all of which are natural extensions of each other. In this final section we consider the collection of possible choices of such families of natural extensions of a fixed base logic.

Throughout the section we fix a regular cardinal $\kappa$, a signature $\Sigma$, an infinite  set $X$ and a logic $l=(\Fm{\Sigma}{X}, {\vdash})$.

We consider the collection $\FP{\Sigma}$ of all filter pairs $(G,i)$ such that $G : \SigmaStr^{op} \to \CLat$. We can give this  collection the structure of a category by defining a morphism $(G',i') \to (G,i)$ to be a natural transformation $t \colon G \to G'$ (note the opposite direction!) such that the following triangle of natural transformations commutes:

 $$\xymatrix{
G \ar[dr]_i \ar[rr]^t &&  G'\ar[dl]^{i'} \\
 & \wp & 
}$$

In fact we will be more interested in $\FP{l}$, the full subcategory of $\FP{\Sigma}$ such that $\Log{X}(G,i)=l$.

We have introduced the reversal of arrows in the definition of $\FP{l}, \FP{\Sigma}$, because in this way morphisms of filter pairs induce translations between their associated logics \emph{in the same direction}: Indeed, a map of logics induces, by taking preimage, a map in the opposite direction between the theory lattices.

In particular the passage to a stronger logic over the same signature means restriction to a smaller theory lattice, which is reflected in the {\bf anti}-isomorphism between the poset of sublattices of powerset lattices and the poset of closure operators from Section \ref{SectionPreliminaries}.

Here is an overview of how directions of morphisms correspond to each other:

 $$\xymatrix{
 \text{logics and translations:} & l \ar[r] & l'  \\
 \text{closure operators:} & C_l \ar[r] & C_{l'} \\
 \text{theory lattices:} & Th(l) & Th(l') \ar[l] \\
 \text{filter pairs:} & \FilPa(l) \ar[r] & \FilPa(l')
 }$$

Here the last reversal of the arrow is purely formal; literally such an arrow is given by lattice maps in the opposite direction. 
 
It is probably helpful, in the following, to keep in mind that one can either think of lattice inclusions and revert arrows or think directly in terms of closure operators and maintain the direction of arrows -- whichever provides a better understanding.


\begin{Obs}
\begin{itemize}

\item It would also be a natural choice to demand an inclusion $i \subseteq i'\circ t$ instead of the equality $i = i'\circ t$, but for the current discussion this would only add redundancy.

    \item The categories $\FP{l}, \FP{\Sigma}$ can be seen as (non-full) subcategories of the category of all $\kappa$-filter pairs, that generalizes the category of all finitary filter pairs introduced in \cite{AMP}.

\end{itemize}

\end{Obs}


The first observation is that the category $\FP{l}$ has a initial object.

\begin{Lem}\label{LemmaImageOfiConsistsOfFilters}
 Let $(G,i)$ be a $\kappa$-filter pair and $X$ a set. Then for every $\Sigma$-structure $M$ and $a \in G(M)$ the set $i_M(a) \in \wp(M)$ is a filter for the logic $\Log{X}(G,i)$.
\end{Lem}
\begin{prooof}
 Consider a $\Sigma$-structure $M$ and an element $a \in G(M)$. Let $\Gamma \cup \varphi \subseteq \Fm{\Sigma}{X}$ be such that $\Gamma \vdash_{\Log{X}(f)} \varphi$ and let $\sigma \colon \Fm{\Sigma}{X} \to M$ be a morphism such that $\sigma(\Gamma) \subseteq i_M(a)$. We need to show that $\sigma(\varphi) \in i_M(a)$.
 
 By the commutativity of the naturality square below, we have  $\sigma^{-1}(i_M(a)) = i_{\Fm{\Sigma}{X}}(G(\sigma)(a))$ for every $a \in G(M)$:
 
 $$\xymatrix{
\Fm{\Sigma}{X} \ar[d]_\sigma &  G(\Fm{\Sigma}{X}) \ar[r]^-{i_{\Fm{\Sigma}{X}}} & \wp(\Fm{\Sigma}{X}) \\
M & G(M) \ar[u]^{G(\sigma)} \ar[r]^-{i_M} & \wp(M) \ar[u]_{\sigma^{-1}} 
}$$
 
 By definition of $\Log{X}(f)$ the hypothesis $\Gamma \vdash_{\Log{X}(f)} \varphi$ means that $\varphi$ is  contained in every set in the image of $i_{\Fm{\Sigma}{X}}$ that contains $\Gamma$. Thus, since $$\Gamma \subseteq \sigma^{-1}\sigma(\Gamma) \subseteq \sigma^{-1}(i_M(a)) = i_{\Fm{\Sigma}{X}}(G(\sigma)(a)),$$ we also have $\varphi \in i_{\Fm{\Sigma}{X}}(G(\sigma)(a)) = \sigma^{-1}(i_M(a))$. Applying $\sigma$ yields $\sigma(\varphi) \in  i_M(a)$.
\end{prooof}

\begin{Prop}\label{PropTerminalObjectInFilterPairsPresentingAFixedLogic}
 The filter pair $\FilPa(l)$ is the initial object of the category $\FP{l}$. In other words, for every filter pair $f=(G',i')$ such that $\Log{X}(f)=l$ there is a unique morphism from $\FilPa(l)$ to $f$.
\end{Prop}
\begin{prooof}
We really construct a terminal object in the opposite category:
 Lemma \ref{LemmaImageOfiConsistsOfFilters} says that for every $M\in \SigmaStr$, $i_{M}(G'(M))\subseteq Fi_{\Log{X}(f)}(M)$. These inclusions form a natural transformation $t \colon G' \to Fi_{l}$ fitting into a commutative triangle with the inclusion $i \colon Fi_{\Log{X}(f)}(M) \subseteq \wp(M)$ of $l$-filters into all subsets. The uniqueness of $t$ simply follows from the fact that $i$ is objectwise injective.
\end{prooof}


\vspace{2mm}

One may ask about further structure or properties of the categories $\FP{\Sigma}$, $\FP{l}$. This would lead to a discussion which is best carried out in the context of general morphisms of filter pairs, and is left for a later work. 

The following consideration shows that to get a meaningful parametrization of the collection of filter pairs that give rise to a fixed logic $l$, we should restrict to the so-called mono filter pairs:

\begin{Obs}\label{RemarkProperClassOfEquivalentFilterPairs}
The isomorphism classes of objects of $\FP{l}$ form a proper class: Given a filter pair $(G,i)$ and a lattice $L$, we can construct a new filter pair $(G^L, i^L)$ defined by $$G^L(A):=L \times G(A)$$ and $$i^L \colon L \times G(A) \stackrel{pr}{\to} G(A)  \stackrel{i_A}{\to} \wp(A).$$
It is clear that the image of $i^L_A$ equals the image of $i_A$, and thus both filter pairs give rise to the exact same family of generalized matrices. Since we can repeat this construction with each member of some proper class of lattices of ever bigger cardinality, there is a proper class of non-isomorphic filter pairs giving back the same generalized matrices.

If we want to see a filter pair as a presentation of a collection of generalized matrices, we might therefore choose to identify two filter pairs $(G,i)$ and $(H,j)$, if the images of the maps $i_A$ and $j_A$ coincide for all $A$. Each equivalence class then has a unique member for which all maps $i_A$ are injective: The filter pair where the lattice consists of the image of $i_A$ and and the natural transformation is the inclusion.
\end{Obs}

In the light of the previous remark, we now concentrate on \emph{mono filter pairs}, i.e. filter pairs $(G,i)$ such that $i_A$ is injective (i.e. a \emph{mono}morphism) for every $A \in \SigmaStr$. These mono filter pairs parametrize the equivalence classes of Remark \ref{RemarkProperClassOfEquivalentFilterPairs}. The full subcategory of $\FP{\Sigma}$ (resp. $\FP{l}$)  whose objects are mono filter pairs will be denoted by $\FPm{\Sigma}$ (resp. $\FPm{l}$).   One sees immediately that this category is actually a pre-ordered {\bf class}, because if both $i_A$ and $i'_A$ in the defining triangle for morphisms are injective, then  $t_A$ is unique and is injective too, for each $A \in obj(\SigmaStr)$. 

Other subcategories that are natural to consider are $\FPi{\Sigma}$ and $\FPi{l}$, where the maps  $i_A$ and $i'_A$ (and thus $t_A$) are in fact {\em inclusions}, $A \in obj(\SigmaStr)$. Obviously $\FPi{\Sigma}, \FPi{l}$ are partially ordered classes and, moreover  $\FPi{\Sigma} \simeq \FPm{\Sigma}$ and  $\FPi{l} \simeq \FPm{l}$.  Dealing directly with $\FPi{\Sigma}, \FPi{l}$ turns easier all the calculations, in fact, is easier to deal first with (arbitrary) infima and (set-sized) suprema in $\FPi{\Sigma}$  -- described ``coordinatewise'' from the results on intersection families recalled in Section 1 --  and then provide the adaptions needed to calculate infima and suprema in $\FPi{l}$. But a direct calculation is provided below: 

\begin{Prop}\label{PropMonoFilterPairsPresentingFixedLogicAreCompleteLattice}
The partially ordered  class  equivalent to 
$\FPm{l}$ admits set sized suprema of nonempty sets.
\end{Prop}
\begin{prooof}
 Let $(G^{r},i^{r})_{r\in R}$ where $R$ is a set. Consider $C^{r}_{A}$ the closure operator over  $A\in \SigmaStr$ determined by $(G^{r},i^{r})$ as, for each $M\subseteq A$,

\[C^{r}_{A}(M)=\bigcap_{a\in G^{r}(A)}\{i^{r}_{A}(a)|\ M\subseteq i^{r}_{A}(a)\}.\]

The closed sets of $C^{r}_{A}$ are exactly the image of $i^{r}_{A}$. Define the operator $C_{A}$ as, for each subset $M$ of $A$,  

\[C_{A}(M):=\bigcup_{N\subseteq M,|N|<\kappa} \bigcap\{N\subseteq X\subseteq A|\ X=C^{r}_{A}(X)\ \forall r\in R\}.\]

It is easy to check that $C_{A}$ is a closure operator. Notice that for each subset $N\subseteq M$ such that $|N|<\kappa$, $C_{A}(N)=\bigcap\{N\subseteq X\subseteq A|\ C^{r}_{A}(X)=X\ \forall r\in R \}$. Then $C_{A}(M)=\bigcup_{N\subseteq M,|N|<\kappa}C_{A}(N)$. Proving the $\kappa$-arity of $C_{A}$ . Now we prove that $C_{A}$ is the supremum of $(C^{r}_{A})_{r\in R}$ . 

Let $M\subseteq A$ and  $N\subseteq M$ where $|N|<\kappa$ such that. Notice that $C^{r}_{A}(N)\subseteq X$ for each subset $N\subseteq X\subseteq A$ such that $C^{r'}_{A}(X)=X$ for all $r'\in R$. So, $C^{r}_{A}(N)\subseteq C_{A}(N)$. Since $C_{A}$ is $\kappa$-ary, we have that $C^{r}_{A}(M)\subseteq C_{A}(M)$. Thus $C^{r}_{A}\leq C_{A}$. Now, let $C$ be a $\kappa$-ary closure operator over $A$ such that $C^{r}_{A}\leq C$ for all $r\in R$. Let $N\subseteq A$ such that $|N|<\kappa$. Let $X$ be a subset of $A$ containing $N$ such that $C(X)=X$. Since $C^{r}_{A}\leq C$ for all $r\in R$, we have that $C_{A}(N)=\bigcap\{X\supseteq N|\  C^{r}_{A}(X)=X\ \forall r\in R\}\subseteq \bigcap\{X\supseteq N|\ C(X)=X\}=C(N)$. Since $C_{A}$ and $C$ are $\kappa$-ary closure operator, then $C_{A}\leq C$. This proves that $C_{A}=\bigvee_{r\in R}C^{r}_{A}$.

Define  the application $G:\SigmaStr\to \Lk$ such that $G(A)$
 is the $\kappa$-lattice of $C_{A}$-closed sets. For a morphism $f:A\to B$ of $\SigmaStr$, $G(f):=f^{-1}$. First notice that for any $r\in R$, and $F$ closed set of $C^{r}_{B}$, then $f^{-1}(F)$ is a closed set of $C^{r}_{A}$. Since $C_{A}$ is the supremum of $C^{r}_{A}$ for all $r\in R$, then, for $F$ closed set of $C_{B}$, $C_{A}(f^{-1}(F))=\bigvee_{r\in R}C^{r}_{A}(f^{-1}(F))=f^{-1}(F)$. Thus $f^{-1}(F)$ is a closed set of $C_{A}$. This proves the functoriality of $G$ and that $(G,i)$ is a mono $\kappa$-filter pair.
 
 We have constructed the closure operator of $G$ at each $\Sigma$-structure as supremum of the closure operators of the $G^r$. This induces inclusions of the theory lattices in the opposite directions, $G(M)\hookrightarrow G^r(M)$ for all $M \in \SigmaStr$, and one readily sees that these form a natural transformation. By the reversal of arrows in $\FPm{l}$, this means $(G^r,i)\leq (G,i)$ for all $r\in R$. As $(G,i)$ was constructed as a pointwise supremum, it is a supremum.
\end{prooof}



\begin{Obs}\label{ObsSeForConjuntoEntaoEReticuladoCompleto}
If $\FPm{l}$ is equivalent to a set, then by Prop. \ref{PropMonoFilterPairsPresentingFixedLogicAreCompleteLattice} and Prop. \ref{PropTerminalObjectInFilterPairsPresentingAFixedLogic} it is equivalent to a complete lattice. In this case we also have a terminal object and arbitrary infima. 
\end{Obs}

{If  $\FPm{l}$ has a terminal object, i.e. a mono $\kappa$-filter pair $(H,j)$ presenting $l$ into which all other filter pair in $\FPm{l}$ map, then we can give a concrete description of the values of this filter pair on free algebras:}

\begin{Lem}\label{LemmaMaximalNatExtensionRestrictsToMaximalNatExtension}
 Let $X \subseteq Y \subseteq Z$ be sets of variables and $l=(\Fm{\Sigma}{X}, {\vdash}_l)$ a logic. Consider the maximal natural extension $l^{+,Z}_\kappa = (\Fm{\Sigma}{Z}, {\vdash}_\kappa^{+,Z})$ of $l$ to the set of variables $Z$. Then $\Log{Y}(\FilPa(l^{+,Z}_\kappa)) = l^{+,Y}_\kappa$, i.e. the restriction of the maximal extension to $\Fm{\Sigma}{Z}$ down to $\Fm{\Sigma}{Y}$ is again the maximal extension.
\end{Lem}
\begin{prooof}
 We know that $l^{+,Z}$ is a conservative extension of $l$, so the $l^{+,Z}$-filters on $\Fm{\Sigma}{X}$ are exactly the $l$-theories, i.e. $\Log{X}(\FilPa(l^{+,Z}_\kappa)) = l$. Thus by Corollary \ref{CorollaryKappaFilterPairPresentingLogicOfCardinalityKappaAlsoYieldsNaturalExtensions} $\Log{Y}(\FilPa(l^{+,Z}_\kappa))$ is a natural extension of $l$ with set of variables $Y$. Since $l^{+,Y}_\kappa$ is the strongest such extension, we have $\Gamma \vdash_{\Log{Y}(\FilPa(l^{+,Z}_\kappa))} \varphi \Rightarrow \Gamma {\vdash^{+,Y}_\kappa} \varphi$. 
 
 For the opposite implication suppose $\Gamma \cup\{\varphi\} \subseteq \Fm{\Sigma}{Y}$ are such that  $\Gamma \vdash^{+,Y}_\kappa \varphi$. Then by definition of the maximal natural extension, there is a $\Gamma' \subseteq \Gamma$ such that $|\Gamma'| \leq \kappa$ and for every substitution $\sigma \colon \Fm{\Sigma}{Y} \to \Fm{\Sigma}{X}$ we have $\sigma(\Gamma') \vdash_l \sigma(\varphi)$. 
 
We need to show that $\Gamma \vdash_{\Log{Y}(\FilPa(l^{+,Z}_\kappa))} \varphi$. Since by Theorem \ref{TheoremFilterPairsGiveConservativeExtensions} \linebreak $l^{+,Z}_\kappa = {\Log{Z}(\FilPa(l^{+,Z}_\kappa))}$ is a conservative extension of ${\Log{Y}(\FilPa(l^{+,Z}_\kappa))}$, this is equivalent to showing $\Gamma \vdash^{+,Z}_\kappa \varphi$, i.e. to showing that  there is a $\Gamma' \subseteq \Gamma$ such that $|\Gamma''| \leq \kappa$ and  for every substitution $\sigma \colon \Fm{\Sigma}{Z} \to \Fm{\Sigma}{X}$ we have $\sigma(\Gamma'') \vdash_l \sigma(\varphi)$. For this we can simply take $\Gamma'':=\Gamma'$ and observe that every such substitution can be restricted to a substitution $\sigma \colon \Fm{\Sigma}{Y} \to \Fm{\Sigma}{X}$, and then we know that $\sigma(\Gamma') \vdash_l \sigma(\varphi)$.
\end{prooof}

\begin{Prop}\label{PropValuesOfInitialMonoFilterPairOnFreeAlgebras}
Let $l:=(\Sigma, X, {\vdash})$ be a logic of cardinality $\kappa$.
{Suppose that $(H,j)$ is a terminal filter pair  in $\FPm{l}$}. Then $H$ is determined on the absolutely free algebras $\Fm{\Sigma}{Y}$ as follows: it takes the value $H(\Fm{\Sigma}{Y}) = \FilPa(l^+_\kappa)(\Fm{\Sigma}{Y})$, the set of filters of the maximal natural extension $l^+_\kappa$ to $\Fm{\Sigma}{Y}$.
\end{Prop}
\begin{prooof}
 We know from Corollary \ref{CorollaryKappaFilterPairPresentingLogicOfCardinalityKappaAlsoYieldsNaturalExtensions} that $H(\Fm{\Sigma}{Y})$ is the set of filters of a natural extension. It is the strongest natural extension $l^+_\kappa$ that has the fewest filters, so if there exists a mono filter pair with the values $\FilPa(l^+_\kappa)(\Fm{\Sigma}{Y})$, for every set $Y$, then these are necessarily also the values of the initial one.
\end{prooof}
\medskip

While, as illustrated by Prop. \ref{PropValuesOfInitialMonoFilterPairOnFreeAlgebras}, the possible values on free algebras are sharply restricted once one knows the logic represented by a mono filter pair, it is harder to say something about non-free algebras. 

For obtaining a precise statement \emph{disregarding} the non-free algebras, we consider a variant of the notion of $\kappa$-filter pair: a \emph{free $\kappa$-filter pair} is a pair $(G,i)$ where $G \colon \SigmaStr_{\rm{free}}^{op} \to \Lk$ is a functor from the category of \emph{absolutely free} $\Sigma$-structures and all endomorphisms to the category of $\kappa$-presentable lattices, and $i$ is a natural transformation exactly as in the definition of $\kappa$-filter pair. Every $\kappa$-filter pair has an underlying free $\kappa$-filter pair, given by restricting the functor part from all $\Sigma$-structures to just absolutely free $\Sigma$-structures. Clearly the associated logics only depend on this restricted filter pair. Indeed, this restriction corresponds to adopting the point of view of a filter pair as a presentation of a family of logics, instead of a whole family of generalized matrices, see Example \ref{ExamplesOfGeneralCFilterPairs}. 

For a fixed logic $l$ of cardinality $\kappa$, we have the categories $\freeFP{l}$ and \linebreak $\freeFPm{l}$ and the restriction functors $\FP{l} \to \freeFP{l}$, resp. $\FPm{l} \to \freeFPm{l}$ which forget about the values at non-free algebras. The map \linebreak $\FPm{l} \to \freeFPm{l}$ is a quotient map, which identifies two mono filter pairs if their values agree for free algebras. Of course $ \freeFPm{l}$ is still a pre-ordered class.

\vspace{0.2cm}

With our final result we give a description of the 
pre-ordered  class $\freeFPm{l}$:

\begin{Teo}\label{TeoLatticeIsoBetweenFreeMonoFilterPairsAndNatExts}
Let $l:=(\Sigma, X, {\vdash})$ be a logic of cardinality $\kappa$ and  $Z$ be a set of cardinality $\kappa$.

Then the  pre-ordered class $\freeFPm{l}$ is equivalent to the poset
of natural extensions of $l$ to $\Fm{\Sigma}{Z}$, ordered by deductive strength, and both are equivalent to complete lattices.
\end{Teo}
\begin{prooof}
Denote the poset of natural extensions of $l$ to $\Fm{\Sigma}{Z}$, ordered by deductive strength, by $\textrm{NatExt}_Z(l)$.

The claimed equivalence is given by the map  $\Log{Z} \colon \freeFPm{l} \to \textrm{NatExt}_Z(l)^{}$ that sends a free mono filter pair presenting $l$ to the associated logic with set of variables $Z$.

It is clear that the map is order preserving, since having more filters means presenting a weaker logic (and the inclusions of lattices become morphisms in the opposite direction in $\freeFPm{l}$). 

{\em The map }$\Log{Z} \colon \freeFPm{l} \to \textrm{NatExt}_Z(l)^{}${\em is surjective}:\\
Let $l'$ be a natural extension of $l$ to $\Fm{\Sigma}{Z}$. The filter pair $\FilPa(l')$ is a mono filter pair. Since $l'$ is a conservative extension of $l$ by assumption and $l'=\Log{Z}(\FilPa(l'))$ (Thm. \ref{TheoremEveryLogicComesFromAFilterPair}) is also a conservative extension of $\Log{X}(\FilPa(l'))$ by Theorem \ref{TheoremFilterPairsGiveConservativeExtensions}, we have $\Log{X}(\FilPa(l'))=l$. So $\FilPa(l') \in \FPm{l}$, and thus for its restriction to free algebras we have $\FilPa(l') \in \freeFPm{l}$. By Theorem \ref{TheoremEveryLogicComesFromAFilterPair} $\Log{Z}(\FilPa(l'))=l'$, which shows surjectivity.

{\em The map }$\overline{\Log{Z}} \colon \freeFPm{l}/\!\cong \ \to \textrm{NatExt}_Z(l)^{}${\em is injective}:\\ 
We show that for a filter pair $(G,i)$ presenting the logic $l$, the value $\Log{Z}(G,i)$ completely determines the values of the filter pair on free algebras $\Fm{\Sigma}{Y}$. Indeed, for a set $Y$ of lower cardinality than $Z$ the consequence relation of $\Log{Y}(G,i)$ is simply the restriction from $\Fm{\Sigma}{Z}$ to $\Fm{\Sigma}{Y}$ by Theorem \ref{TheoremFilterPairsGiveConservativeExtensions}, so the filters are determined up to isomorphism by those on $\Fm{\Sigma}{Z}$. On the other hand, for a set $Y$ of bigger cardinality than $Z$, the logic $\Log{Y}(G,i)$ will be a natural extension of $\Log{Z}(G,i)$ by Corollary \ref{CorollaryKappaFilterPairPresentingLogicOfCardinalityKappaAlsoYieldsNaturalExtensions}, but the latter has a \emph{unique} natural extension by Cor. \ref{CorollaryUnderCardinalityConditionsNaturalExtensionsAreUnique}, so this is also completely determined by $\Log{Z}(G,i)$.

We thus have an isomorphism of partially ordered classes  $\overline{\Log{Z}} : \freeFPm{l}/\!\cong \ \overset{\cong}\to \textrm{NatExt}_Z(l)^{}$. But since there is only a \emph{set} of natural extensions, by Remark \ref{ObsSeForConjuntoEntaoEReticuladoCompleto} both are complete lattices.
\end{prooof}
\medskip

\begin{Cor} \label{natextlattice} 
The set of natural extensions of a logic $l$ of regular cardinality $\kappa$ with respect to a fixed set of variables $Z$ with cardinality $\kappa$, ordered by deductive strength, is a complete lattice.
\end{Cor}

We conclude by remarking that the results of this article suggest to view a $\kappa$-filter pair as a presentation of a logic \emph{together with a chosen family of natural extensions}. In fact, the notion of free mono filter pair captures precisely that.

The view of finitary filter pairs as presentations of a logic, suggested in \cite{AMP}, remains as valid as before: by Cintula and Noguera's uniqueness result, Cor. \ref{CorollaryUnderCardinalityConditionsNaturalExtensionsAreUnique}, for a finitary logic there is a unique natural extension to every set of variables, and hence the lattices of Theorem \ref{TeoLatticeIsoBetweenFreeMonoFilterPairsAndNatExts} are trivial. Thus this is a genuinely new aspect arising for $\kappa$-filter pairs.

\section{Final remarks}\label{SectionFinalRemarks}

We have introduced $\kappa$-filter pairs and raised the question, which different $\kappa$-filter pairs give rise to the same fixed logic. After restricting to mono filter pairs to make it into a tractable question, this is equivalent to the question into which different coherent systems of generalized matrices the given logic can fit.

This seems to be a difficult question, and an answer in full generality at the moment seems elusive. We have, however, completely solved the "free algebra part" of the problem, in terms of the natural extensions of the logic. 
We expect that a full answer would provide insights on the prospects for generalized matrices described in \cite{FontGmatrix} (see also \cite[Chapter 5]{Fon}), many of which have not been developed very far.

There is one respect which we didn't mention, in which our solution for the free algebra part already gives interesting information: As laid out in \cite[Section 4]{FontGmatrix}, generalized matrices can be understood as models of Gentzen systems, and we can understand our result as saying that the different Gentzen systems describing a given logic correspond precisely to the natural extensions of the logic. We thank the editor for her or his remark pointing us into this direction.

The main interest in $\kappa$-filter pairs is that they can be used to treat infinitary logics along the lines of \cite{AMP}. Most results of loc. cit. carry over, and the prospects listed for finitary filter pairs in the final section of loc. cit continue to be make sense and be interesting. The extra flexibility of allowing logics of higher cardinalities can be used to speak about logics which have an algebraic semantics in generalized quasivarieties, via the congruence filter pairs of Section \ref{SectionFilterPairs}.

The present article laid technical groundwork for this, and we intend to follow up with concrete applications (some of which were hinted at in Section \ref{SectionFilterPairs}) and further steps in the long-term project laid out in \cite{MaPi1}, \cite{MaPi2}, of establishing local-global principles in logic, setting up a representation theory of logics and giving applications in remote algebraization.

\end{document}